\documentclass[]{scrartcl}

\usepackage[USenglish]{babel}
\usepackage[utf8]{inputenc}

\usepackage{geometry}
\usepackage{mathtools}
\usepackage{amssymb}
\usepackage{amsthm}
\usepackage{graphicx}
\usepackage[short, nocomma]{optidef} 
\usepackage{booktabs} 
\usepackage{array}
\usepackage{tikz}
\usetikzlibrary{positioning, shapes.geometric, decorations.pathreplacing,backgrounds,intersections}
\usepackage{pgfplots}
\usepgfplotslibrary{fillbetween}
\pgfplotsset{compat=1.18}
\usepackage{nicefrac}

\usepackage[labelformat=simple]{subcaption}

\usepackage[shortlabels]{enumitem} 

\usepackage[linesnumbered,ruled,vlined]{algorithm2e}
\SetArgSty{text}
\newcommand{\atcp}[1]{\tcp*[f]{#1}} 

\usepackage{authblk}

\title{Fare Structure Design in Public Transport\footnote{This work was partially supported by the European Union's Horizon 2020 research and innovation programme [Grant 875022] and the Federal Ministry of Education and Research [Project 01UV2152B] under the project sEAmless SustaInable EveRyday urban mobility (EASIER).}
}
\author[1,2]{Anita Schöbel}
\author[1]{Reena Urban}
\date{}
\affil[1]{University of Kaiserslautern-Landau (RPTU), Kaiserslautern, Germany}
\affil[2]{Fraunhofer-Institute for Industrial Mathematics ITWM, Kaiserslautern, Germany}
\affil[ ]{\{anita.schoebel, reena.urban\}@math.rptu.de}

\newcommand{\PreserveBackslash}[1]{\let\temp=\\#1\let\\=\temp}
\newcolumntype{C}[1]{>{\PreserveBackslash\centering}p{#1}}
\newcolumntype{R}[1]{>{\PreserveBackslash\raggedleft}p{#1}}
\newcolumntype{L}[1]{>{\PreserveBackslash\raggedright}p{#1}}

\newcommand{\NN}{\mathbb{N}} 
\newcommand{\RR}{\mathbb{R}} 
\newcommand{\defeq}{\mathrel{\vcentcolon=}} 
\newcommand{\eqdef}{\mathrel{=\vcentcolon}} 
\newcommand{\cW}{\mathcal{W}} 
\newcommand{\cZ}{\mathcal{Z}} 
\newcommand{\cC}{\mathcal{C}} 
\newcommand{\cO}{\mathcal{O}} 
\newcommand{\dist}{\textsf{\textup{dist}}} %
\newcommand{\weightedMedian}{\textsf{\textup{w-median}}} 
\DeclareMathOperator{\interior}{int} 
\DeclareMathOperator{\cut}{cut} 
\newcommand{\BipartiteSubgraph}{\textup{\textsc{Bipartite Subgraph}}} 
\newcommand{\Multicut}{\textup{\textsc{Multicut}}} 
\newcommand{\MultipleChoiceLinearProgrammingProblem}{\textsc{Multiple Choice Linear Programming Problem}}
\newcommand{\ZDMA}{\textup{\textsc{ZD-MA}}} 
\newcommand{\ZDMC}{\textup{\textsc{ZD-MC}}} 
\newcommand{\ZDSA}{\textup{\textsc{ZD-SA}}} 
\newcommand{\ZDSC}{\textup{\textsc{ZD-SC}}} 
\newcommand{\ZDXY}{\textup{\textsc{ZD-XY}}} 

\usepackage{abstract}

\usepackage[hidelinks]{hyperref}
\usepackage{cleveref}      	
\theoremstyle{plain}
\newtheorem{theorem}{Theorem}
\newtheorem{lemma}[theorem]{Lemma}

\newtheorem{corollary}[theorem]{Corollary}
\theoremstyle{definition}
\newtheorem{definition}[theorem]{Definition}
\newtheorem{example}[theorem]{Example}
\newtheorem{remark}[theorem]{Remark}

\begin{document}
	
\maketitle

\begin{abstract}
	Fare planning is one among several steps in public transport planning.
	Fares are relevant for the covering of costs of the public transport operator, but also affect the ridership and the passenger satisfaction.
	A fare structure is the assignment of prices to all paths in a network.
	In practice, often a given fare structure shall be changed to fulfill new requirements, meaning that a new fare strategy is desired. This motivates the usage of prices of the former fare structure or other desirable prices as reference prices.
	In this paper, we investigate the fare structure design problem that aims to determine fares such that the sum of absolute deviations between the new fares and the reference prices is minimized. Fare strategies that are considered here are flat tariffs, affine distance tariffs and zone tariffs.
	Additionally, we regard constraints that ensure that it is not beneficial to buy a ticket for a longer journey than actually traveled (no-elongation property) or to split a ticket into several sub-tickets to cover a journey (no-stopover property).
	Our literature review provides an overview of the research on fare planning.
	We analyze the fare structure design problem for flat, distance and zone tariffs, pointing out connections to median problems.
	Further, we study its complexity which ranges from linear-time solvability to NP-complete cases.
\end{abstract}

\textbf{Keywords:} public transport, fare planning, fare structure design, modeling, complexity, linear programming

\newpage

\section{Introduction}

\subsection{Fare Structures}
Fares for public transport usage affect the number of people traveling by public transport as well as the passenger satisfaction, and they are important for covering the costs of the public transport operator.
A variety of fare structures (which assign prices to paths) are implemented worldwide, each with a different focus and purpose.
These can be flat or differential fares, where the latter are further distinguished according to whether they depend on a distance, duration, time or quality component of the journey \cite{Fleishman1996,Schmoecker2016,Czerlinski2021}.
In this paper, we focus on flat, distance and zone tariffs as they are very popular in many countries.

In a \emph{flat tariff}, every ticket has the same price.
This has the advantage that it is very easy to understand, but on the other hand it is often perceived as unfair because passengers with a short journey pay the same price as passengers with a long journey.

At the other end of the spectrum, there are \emph{distance tariffs}, which can depend on the actual distance traveled in the network (network distance tariff) or on the Euclidean distance between the start and the end station of the journey (beeline distance tariff).
In this paper, we consider \emph{affine} distance tariffs, which are composed of a price per kilometer and an additional base amount for the network distance or Euclidean distance, respectively.
Hence, the fare is directly linked to the length of a journey.
The rising popularity of mobile tickets and smart cards has led to an increased interest in distance tariffs because check-in/check-out systems can be used to determine the distance.

\emph{Zone tariffs} group stations to zones and set fares based on the traversed zones.
While for affine distance tariffs all station pairs have an individual price, there are fewer different price levels in zone tariffs.
This paper concerns a \emph{counting-zones pricing}, which means that the area of the public transport operator is divided into zones and the price of a journey is determined according to the number of zones that are traversed.
Hence, the number of zones can be seen as an approximate distance measure.
We distinguish between the option that a zone is counted each time that it is entered (multiple counting) and that the total number of different zones is relevant (single counting).
Furthermore, we consider zone tariffs with and without requiring connected zones.
In practice, zones are not always required to be connected which leads to separate parts of the same zone that can only be reached from each other by traveling through other zones.

As two important criteria in practice for a fare structure to be fair and consistent, we regard the no-elongation property and the no-stopover property introduced in \cite{Schoebel2020, Schoebel2022}.
While the no-elongation property ensures that it is not beneficial to buy a ticket for a longer journey that contains the actual path, the no-stopover property makes sure that it is not beneficial to split a ticket into several tickets for subpaths.
When designing fare structures, we deem these properties important because, on the one hand, they prevent the undercutting of tariffs and, on the other hand, they make it easier to buy tickets because passengers do not need to check whether other tickets are cheaper than that for their actual journey.
Note that in practice, these properties are not always fulfilled.
For insights into horizontal and vertical equity in fare planning, we refer to \cite{Rubensson2020} which examines the geographical and distributional fairness of flat, distance and zone tariffs.

The question that we address in this paper is how to design fare structures.
Our objective is to design a fare structure with fares that are close to given reference prices for the origin-destination pairs under the assumption of a fixed demand.
The notion of reference prices has been introduced in \cite{Hamacher1995} and has further been used in \cite{Babel2003, Hamacher2004,Paluch2013, Paluch2017,Galligari2017}.
The reference prices can, for example, be chosen as prices that are considered fair, or in order to increase the acceptance of the fares and attract more passengers, or as prices of a former fare structure to maintain the prices passengers are used to.
In this paper, an objective function is implemented that minimizes the weighted sum of absolute deviations from reference prices.

\subsection{Literature Review}
In the literature, various optimization models are presented with the aim to determine a fare structure.
The maximization of revenue, demand, fairness or social welfare can be objectives pursued in fare planning.
Thereby, passenger choices are considered to varying degrees: from fixed paths, maybe with a reference price, a willingness to pay or  elasticities to route choice subproblems and equilibrium constraints.
We first review publications that perform simulation studies and regression analysis to examine \emph{given} fare structures as well as publications on cheapest paths.
Our main focus then is on the fare \emph{planning} literature that applies optimization models and techniques to determine optimal fares with respect to the above mentioned objectives.
We separately give an overview of literature on distance- and zone-based fare structures.

\paragraph*{Examination of Given Fare Structures}
A simulation study to evaluate the impact of implementing a distance or a zone tariff on ridership and revenue is conducted in \cite{Gattuso2007}.
Similarly, \cite{Chin2016} performs computational experiments changing from a flat to a distance tariff in a congested network.
Besides the impact on ridership and profit, also implications for different segments of the population are investigated.
In \cite{Maadi2020}, the effects of changes from a zone to a distance tariff on route choice, total travel time, total fares paid and the amount of walking are illustrated in a simulation study.
In \cite{Czerlinski2021}, the authors group fare structures of eleven cities in Poland into flat, distance-, quality-, time- and zone-based fare structures and apply regression analysis to evaluate by which function type (linear, power, polynomial, logarithmic or exponential in the respective unit) they are described best.

The computation of cheapest paths is considered for distance tariffs in a railway context in \cite{MuellerHannemann2006} and for zone tariffs in \cite{Delling2015,Delling2019}.
In \cite{Euler2019,Euler2024}, the so-called ticket graph is presented which models transitions between tickets via transition functions over partially ordered monoids and allows the design of an algorithm for finding cheapest paths in fare structures that do not have the subpath-optimality property. 
The cheapest ticket problem as well as the no-elongation property and the no-stopover property are studied in \cite{Schoebel2020,Schoebel2022} for distance- and zone-based fare structures.

\paragraph*{Fare Planning for Distance-based Fare Structures}
For distance-based tariffs, \cite{Daskin1988} presents a quadratic model for determining a base amount, a price per kilometer and a price per transfer maximizing the revenue. 
As an increasing fare reduces the demand, additional constraints for lower bounds on the demand as well as lower and upper bounds on the fares are applied.
In \cite{Paluch2013, Paluch2017, Hoshino2018}, a distance-based fare structure with two fare levels or an arbitrary but fixed number of fare levels, respectively, based on the number of traversed edges/stations is determined for a line.
While \cite{Paluch2013, Paluch2017} minimize the weighted sum of squared deviations from reference prices, which are given by a granular fare structure with a fare level for each number of traversed edges, \cite{Hoshino2018} maximizes the revenue or the demand while in each case the other one is kept fixed.
In \cite{Yook2014}, a bi-level approach for maximizing the demand that is solved by a genetic algorithm is presented. 
The upper level problem determines distance fares composed of a base amount, a mileage price and surcharges for premier services, while passengers choose routes that minimize their generalized user costs including various time components in the lower level problem.
Non-linear beeline distance fares are designed in \cite{Huang2016} while simultaneously optimizing the service frequency.
The problem is modeled as a three-player game between the transport authority, the transit enterprise and the passengers. 
The authors show that the problem is NP-hard and propose an artificial bee colony algorithm to solve it.

\paragraph*{Fare Planning for Zone-based Fare Structures}
Research on the design of zone-based fare structures has started with \cite{Hamacher1995,Schoebel1996,Hamacher2004}.
They introduce the objective of minimizing the deviation from given reference prices as applied in this paper.
The sum of absolute and squared deviations is considered, but the main focus lies on minimizing the maximum absolute deviation, where passengers minimize the number of traversed zones.
An arbitrary pricing, where the price between each pair of zones is set individually, is applied as well as a counting-zones pricing with multiple counting.
While the prices can be determined by an explicit formula once the zones are given, the integrated problem of determining zones and prices simultaneously is shown to be NP-hard for the objective of minimizing the maximum absolute deviation for all fixed numbers of zones greater than or equal to three.
A greedy, a clustering and a spanning-tree heuristic for the problem of determining zones are presented.
These results are expanded by \cite{Babel2003} which considers the zone tariff design problem with arbitrary pricing, connected zones and minimizing the maximum absolute deviation.
The paper further investigates the complexity regarding NP-completeness and polynomial cases on special graph structures.
In \cite{Pratelli2004}, the problem of finding zones and prices is formulated as a bilevel program that is solved with a simulated annealing heuristic.
The objective is to minimize a weighted sum of revenue increase and decrease.
Furthermore, the authors of \cite{TavaresPereira2007} point out that districting problems occur in the context of different application areas, in particular, as the problem of finding zones for a zone-based fare structure.
The districting problem is modeled as a multicriteria problem, and a local search evolutionary algorithm is proposed to solve it.
In \cite{Kohani2013}, an integer program is presented with the aim to determine zones for fixed prices given as a base amount and a price per zone that minimize the maximum or average absolute deviation between zone fares and reference prices, which are given by former distance fares.
The constructed zones need not be connected.
Also for given prices, \cite{Conejero2014} searches for a zone partition with connected zones along a linear graph that maximizes the revenue, where the fares are determined according to a counting-zones pricing.
A tree-based model is used to enumerate the options for the zones.
\cite{Yang2020} faces a similar problem and solves it with a local search method moving zone borders.
The prices are then chosen as average former ticket costs between pairs of zones.
Another local search heuristic with tabu search is developed by \cite{Galligari2017} for the connected zone design problem with given prices on a general graph. 
Fares are computed according to a counting-zones pricing and the objective function is to minimize the maximum absolute deviation from reference prices.
The local search heuristic improves the found solutions over the heuristics from \cite{Hamacher2004}.
The authors of \cite{Azadian2018} encounter the problem of finding zones and setting prices (arbitrary pricing) in the context of air cargo.
The goal is to maximize the revenue.
The problem is solved to optimality by Benders decomposition, and a branch-and-bound method is developed that outperforms the Benders decomposition.
The theoretical research of \cite{Hamacher2004} and \cite{Babel2003} is further extended by \cite{Otto2017}.
The authors model different zone-based fare strategies, namely with connected or ring zones and with counting-zones pricing (single counting), cumulative pricing or maximum pricing.
The objective function maximizes the revenue considering the willingness to pay of the customers in the constraints.
For these problems, mixed-integer linear programming formulations are provided and the complexity is analyzed.
Recently, \cite{Mueller2022} tackled the problem of determining connected zones with a counting-zones pricing (multiple counting) by modeling the problem in the dual graph.
An integer programming formulation and a heuristic are developed during which the prices are kept fixed.
Iterating over a list of price options, the goal is to find the zone tariff that maximizes the revenue or demand. 
The authors propose an option to enforce certain spatial patterns such as rings and stripes.

\paragraph*{General Fare Planning}
A very general model for fare planning that can deal with different fare strategies and objectives is proposed by \cite{Borndoerfer2012}.
Also monthly and reduced fares can be incorporated in the nonlinear optimization model based on a discrete route choice model.
In \cite{Schiewe2024a}, a bi-objective model suitable for various fare strategies that maximizes the revenue as well as the number of passengers is presented.
Additional specialized solution methods are provided for flat and affine distance tariffs.

\subsection{Our Contribution}
We expand the theoretical literature on fare structure design problems for flat, affine distance and zone tariffs with the objective to minimize the weighted sum of absolute deviations from reference prices. 
We identify the flat and affine distance tariff design problem as median problems and thus show the solvability in linear time.
Further, we analyze and compare different variants of the zone tariff design problem that appear in the literature.
The relationships of the optimal objective function values are summarized in \Cref{fig:relations-overview}.
Further, we investigate the complexity of the zone tariff design problem.
In particular, we show NP-hardness in general, and develop a polynomial algorithm for the price-setting subproblem with fixed zones satisfying the no-elongation property.
An overview of the complexity results is given in \Cref{tab:complexity-results}.

The remainder of the paper is structured as follows: 
In \Cref{sec:general-problem-description}, we introduce the general setting and problem description. \Cref{sec:flat-tariff} covers the results on the flat tariff design problem. 
The determination of distance tariffs, which are in this paper affine network and beeline distance tariffs, is treated in \Cref{sec:distance-tariff}.
In \Cref{sec:zone-based}, we formulate the zone tariff design problem, examine its variants and analyze its complexity.
Finally, we provide a mixed-integer linear programming formulation for the zone tariff design problem.
\Cref{sec:conclusion} contains the conclusion.

\section{General Problem Description}
\label{sec:general-problem-description}
In this section, we introduce the basic knowledge, terminology and notation necessary for the later analysis.
Some definitions have previously been introduced in \cite{Schoebel2022}.

A public transport network (PTN) $G=(V,E)$ is an undirected, connected, simple graph.
The node set $V$ represents a set of stops or stations, and the edge set $E$ represents the direct connections between them.
A subset of nodes $Z\subseteq V$ is called connected if its induced subgraph~$G[Z]$ is connected.
A path is a finite sequence $(v_1,e_1,v_2,\ldots,e_{n-1},v_n)$ of nodes and edges.
It is simple if no edge occurs more than once, and it is elementary if no node occurs more than once.
We denote the set of all paths in the PTN by $\cW$.
In public transport, paths are often elementary.
Nevertheless, if not stated explicitly, the results of this paper hold for arbitrary paths.
Because we assume that a PTN does not have parallel edges or loops, a path is uniquely determined by its nodes, which allows to denote it as a sequence of its nodes $(v_1,\ldots,v_n)$.
The PTN can be used to model railway, tram, or bus networks.
In the following, we call the nodes of the PTN stations, even if bus networks with stops are under consideration.

\begin{definition}[Fare structure \cite{Schoebel2022}]
	Let a PTN be given, and let $\cW$ be the set of all paths in the PTN.
	A \emph{fare structure} is a function $\pi\colon \cW \rightarrow \RR_{\geq 0}$ that assigns a price to every path in the PTN.
\end{definition}

A \emph{fare strategy} stipulates requirements (constraints) on a fare structure.
A fare structure $\pi$ is of a certain fare strategy if $\pi$ satisfies the corresponding strategy constraints.
In the following, we study the design of fare structures regarding flat, affine distance, affine beeline and counting zones fare strategies.

After deciding on a desired fare strategy, we take
\begin{itemize}
	\item a PTN $(V,E)$ as an undirected, connected, simple graph,
	\item a set of OD pairs $D\subseteq (V \times V) \setminus \{(v,v): v\in V\}$,
	\item reference prices $r_d\in \RR_{\geq 0}$ and passenger numbers $t_d \in \NN_{\geq 1}$ with fixed paths $W_d \in \cW$ (for example, time-shortest paths) for all OD pairs $d\in D$
\end{itemize}
as input to the corresponding fare structure design problem.
Depending on the desired fare strategy, more specific input like distances between stations or a limit on the number of zones to be established might be necessary as we will see in the sections dedicated to the different fare strategies.

In the context of this paper, the fare structure design problem is formally defined as follows:
\begin{definition}[Fare structure design problem]
	Given a PTN $(V,E)$, a set of OD pairs $D$ with reference prices $r_d \in \RR_{\geq 0}$, passenger numbers $t_d\in \NN_{\geq 1}$ and fixed paths $W_d\in \cW$ for all $d\in D$ as well as potentially specific input depending on the desired fare strategy, determine a fare structure $\pi$ regarding the desired fare strategy that minimizes the sum of absolute deviations from the reference prices, this means it minimizes $\sum_{d\in D} t_d \vert r_d - \pi(W_d) \vert$.
\end{definition}

In addition, to establish fairness and consistency, we want to design fare structures that satisfy the no-elongation property and the no-stopover property introduced in \cite{Schoebel2020,Schoebel2022}.
For that, let $(v_1,\ldots, v_n)\in \cW$ be a path given as a sequence of nodes $v_1,\ldots, v_n \in V\!$.
\begin{definition}[No-elongation property and no-stopover property \cite{Schoebel2022}]
	\label{definition:properties}
	Let a PTN be given, and let $\cW$ be the set of all paths in the PTN.
	\begin{enumerate}
		\item A fare structure $\pi$ satisfies the \emph{no-elongation property} if 
		\[{\pi((v_1,\ldots,v_{n-1}))\leq \pi((v_1,\ldots,v_n))}\]
		for all paths $(v_1,\ldots,v_n) \in \cW$, $n\geq 2$.
		\item A fare structure $\pi$ satisfies the \emph{no-stopover property} if 
		\[ \pi((v_1,\ldots,v_n)) \leq \pi((v_1,\ldots,v_i))+\pi((v_i,\ldots,v_n)) \]
		for all paths $(v_1,\ldots,v_n) \in \cW$, $n\geq 3$, and intermediate stations $v_i$ with ${i \in \{2,\ldots,n-1\}}$.
	\end{enumerate}
\end{definition}
The no-elongation property ensures that the fare for a path is not allowed to be cheaper than the fare for any sub-path.
The no-stopover property ensures that it is not beneficial to split a ticket into several ones.
This means that buying several tickets to cover a path or making a stopover never decreases the total fare that has to be paid for a path.
From the perspective of the passengers, the properties yield transparency and make the fare structure easier to understand.
As shown in \cite[Thm.~1]{Schoebel2022}, it is cheapest to buy a ticket for exactly the path intended if both properties are satisfied.
From the perspective of a public transport operator this yields consistency in the sense that passengers cannot undercut the fares.
This is, for example, relevant to correctly estimate the revenue.

The fare structure design problem minimizing the sum of absolute deviations from reference prices is closely related to the weighted median problem as we show in the following sections.
We therefore recall the definition of a weighted median.
\begin{definition}[Weighted median]
	\label{definition:weighted-median}
	Given an index set $D$, the set of weighted medians of numbers~$(r_d)_{d\in D}$ with weights~$(t_d)_{d\in D}$ denoted by $\weightedMedian_{d\in D}(r_d,t_d)$ contains all values~$\bar{p}$ that satisfy
	\begin{equation}
		\label{eq:weighted-median}
		\sum_{d\in D:\, r_d < \bar{p}} t_d \leq \frac{\sum_{d\in D} t_d}{2} \quad \text{and} \quad \sum_{d\in D:\, r_d > \bar{p}} t_d \leq \frac{\sum_{d\in D} t_d}{2}.
	\end{equation}
\end{definition}
We use the shorthand notation $\weightedMedian(D)$ if the numbers $(r_d)_{d\in D}$ and weights $(t_d)_{d\in D}$ are clear from the context.
Instead of taking a weighted median one can equivalently consider the median of the values $(\underbrace{r_d,\ldots,r_d}_{t_d \text{ times}}: d\in D)$.
Also note that $\weightedMedian_{d\in D}(r_d,t_d)$ can consist of a single value or, because \eqref{eq:weighted-median} is a convex condition, a whole interval $[\bar{p}_1,\bar{p}_2]$, where $\bar{p}_1$ satisfies the first inequality of~\eqref{eq:weighted-median} with strict inequality and the second with equality, and vice versa for $\bar{p}_2$.
Then $\bar{p}_1$ is called the lower weighted median and $\bar{p}_2$ the upper weighted median.

In the following sections, we consider the fare structure design problems regarding flat, distance and zone fare strategies.

\section{Flat Tariffs}
\label{sec:flat-tariff}
The simplest fare strategy is the flat tariff, in which all paths are assigned the same fixed price.
On the one hand, it is easy to understand and to apply.
On the other hand, it can be perceived as unfair because short trips are as expensive as long trips.

\begin{definition}[Flat tariff \cite{Schoebel2022}] 
	Let a PTN be given, and let $\cW$ be the set of all paths in the PTN.
	A fare structure $\pi$ is a \emph{flat tariff} w.r.t.\ the fixed price $f \in \RR_{\geq 0}$ if $\pi(W)=f$ for all~$W \in \cW$.
\end{definition}

\begin{definition}[Flat tariff design problem]
	Given are a set of OD pairs $D$ as well as a reference price $r_d\in \RR_{\geq 0}$ and the number of passengers $t_d \in \NN_{\geq 1}$ for all $d\in D$.
	The \emph{flat tariff design problem} aims to determine a price $f\in \RR_{\geq 0}$ such that $\sum_{d\in D} t_d \vert r_d - f\vert$ is minimized.
\end{definition}

If we drop the non-negativity requirement $f\geq 0$ from the flat tariff design problem, we obtain a \emph{weighted median problem} as used in statistics (e.g., \cite{Gurwitz1990}), or, equivalently, a \emph{one-dimensional 1-median problem}, where especially the two-dimensional version is well known as \emph{Weber problem}, as used in location theory (e.g., \cite{Plastria1995,DKSW01}).

The function $\varphi\colon \RR \to \RR_{\geq 0}, f \mapsto \sum_{d\in D} t_d \vert r_d - f\vert$ is a continuous, piece-wise linear and convex function, for which it is known that the set of optimal solutions to the weighted median (or one-dimensional 1-median) problem is equal to the set of weighted medians $\weightedMedian_{d\in D}(r_d,t_d)$ (e.g., \cite{Gurwitz1990}).
Hence, for all optimal solutions $f^{\ast}$, it holds that $f^{\ast}\geq \min\{r_d: d\in D\}$, and there is always an optimal solution $f^{\ast}$ with $f^{\ast}\in \{r_d: d\in D\}$.
Because $\min\{r_d: d\in D\}\geq 0$, the non-negativity requirement of the flat tariff design problem is not necessary and can be omitted.
Therefore, we can regard the flat tariff design problem as a weighted median problem.
In case that the weighted median is not unique, the lower median $f_1$ leads to lower prices for the passengers, whereas the operator generates a higher income by implementing the price $f_2$ of the upper median.
These two values as well as all $f\in [f_1,f_2]$ yield the same objective function value because we consider the sum of absolute deviations.

The following well-known linear programming (LP) formulation uses a variable $f\in \RR$ for the fixed price and an auxiliary variable $y_d\in \RR$ for all~$d\in D$ to linearize the objective function:
\begin{equation}
	\label{lp:flat}
	\begin{alignedat}{3}
		&\min_{f,y_d} 	& \sum_{d\in D} t_d y_d& &&\\
		&\text{ s.t. } 	& r_d-f &\leq y_d &\quad& \textup{for all } d\in D\\
		&				& f-r_d &\leq y_d && \textup{for all } d\in D\\
		&				& f		&\geq 0	  &&\\
		&				& f,y_d		&\in \RR	  &&\textup{for all } d\in D.
	\end{alignedat}
\end{equation}

Apart from solving the LP formulation in polynomial time, also specialized solution methods for finding a (weighted) median can be used. 
A linear time selection method, which is in particular able to find a median, is developed by \cite{BlumEtAl73}.
Several methods to compute a weighted median in linear time are reviewed by \cite{Gurwitz1990}.
Consequently, the flat tariff design problem can be solved in linear time $\cO(\vert D\vert)$.

From \cite[Thm.~12]{Schoebel2020}, we know that a flat tariff always satisfies the no-elongation property and the no-stopover property.

\section{Distance Tariffs}
\label{sec:distance-tariff}

Distance-based fare structures are widely perceived as fair because the ticket price correlates with the distance between the start and the end station.
In a network distance tariff, the length of the actual path in the PTN is used to determine the ticket price, whereas in a beeline distance tariff, the Euclidean distance (``beeline distance'') between the start and the end station is used.
In particular, in a beeline distance tariff, the price of a ticket is independent of the actual path but depends just on its start and end station.

Let a PTN $(V,E)$ with edge lengths $l_e\in \RR_{> 0}$ for all edges~$e\in E$ and paths~$W_d\in \cW$ for all OD pairs $d\in D$ be given.
Further, we assume that the stations $V$ of the PTN are embedded in the plane such that the Euclidean distance $\dist_2$ between every pair of stations can be computed.
Let a path $W \in \cW$ have the start station~$v_1 \in V$ and the end station $v_2\in V\!$. 
We set $l(W) \defeq \sum_{e\in E(W)} l_e$ in a network distance tariff and $l(W) \defeq \dist_2(v_1,v_2)$ in a beeline distance tariff.

\begin{definition}[Affine distance tariff \cite{Schoebel2022}]
	Let a PTN be given, and let $\cW$ be the set of all paths in the PTN.
	Let $l\colon \cW \to \RR_{\geq 0}$ be a function, which determines the network distance or the beeline distance of a path (see above).
	A fare structure~$\pi$ is an \emph{affine distance tariff} w.r.t.\ a price per kilometer $p \in \RR_{\geq 0}$ and a base amount $f\in \RR_{\geq 0}$ if $\pi(W)=p\cdot l(W) +f$ for all~${W \in \cW}$.
\end{definition}

For the optimization of affine distance tariffs, we consider that each OD pair $d\in D$ travels along a fixed path $W_d\in \cW$.
Hence, each OD pair is associated with a distance $l_d \defeq l(W_d)$ based on the function $l$ measuring the distance.

\begin{definition}[Affine distance tariff design problem]
	Given are a set of OD pairs $D$ as well as a reference price $r_d\in \RR_{\geq 0}$, the number of passengers $t_d \in \NN_{\geq 1}$ and distances $l_d \in \RR_{\geq 0}$ for all OD~pairs~$d\in D$.
	The \emph{affine distance tariff design problem} aims to determine a price per kilometer $p\in \RR_{\geq 0}$ and a base amount~$f\in \RR_{\geq 0}$ such that $\sum_{d\in D} t_d \vert r_d - (p\cdot l_d + f)\vert$ is minimized.
\end{definition}

Note that in contrast to the flat tariff design problem, for distance tariffs neither $p\geq 0$ nor $f\geq 0$ is ensured by non-negative reference prices $r_d\geq 0$ for all $d\in D$.

When we allow negative values for $p$ and $f$, i.e., $p, f \in \RR$, the affine distance tariff design problem is a \emph{least absolute deviations regression} with the linear expression $p\cdot x + f $ as used in statistics (e.g., \cite{Karst1958,Wagner1959}), or equivalently, a \emph{1-median-line location problem with vertical distances} as used in location theory (e.g., \cite{Megiddo1983,Schoebel1999}).
This means, one searches for a line \[L_{p,f} \defeq \{(x,y)\in \RR^2: y=p\cdot x+f\}.\] 

In \Cref{prop:distance-finite-candidate-set}, we derive a finite dominating set for the affine distance tariff design problem.
We say that a reference price $r_d$ for some OD pair $d\in D$ is met exactly by a distance tariff with price per kilometer~$p$ and base amount~$f$ if $r_d = p\cdot l_d +f$, i.e., the line $L_{p,f}$ passes through the point $(l_d,r_d)$.
\begin{theorem}
	\label{prop:distance-finite-candidate-set}
	There is always an optimal solution $(p^{\ast},f^{\ast})$ to the affine distance tariff design problem such that one of the following holds:
	\begin{itemize}
		\item two reference prices are met exactly,
		\item one reference price is met exactly and, additionally, $p^{\ast}=0$ or $f^{\ast}=0$,
		\item $p^{\ast}=0$ and $f^{\ast}=0$.
	\end{itemize}
\end{theorem}
\begin{proof}
	We regard the affine distance tariff design problem as 1-median-line location problem with the additional requirement that $p,f\geq 0$.
	For this proof, we adopt the dual interpretation from \cite[Sec.~2.2]{Schoebel1999}.
	We consider the transformation $T$ that maps a point $(l,r)\in \RR^2$ to a line $T((l,r)) \defeq L_{-l,r}$ and a non-vertical line $L_{p,f}$ to a point $T(L_{p,f}) \defeq (p,f)$.
	The space of the transformed points and lines is called \emph{dual} space.
	We call the original space the \emph{primal} space.
	An example is given in \Cref{fig:distance-dual-space-example}.
	The vertical deviation between a point~$(l,r)$ and a line~$L_{p,f}$ in the primal space is the same as the vertical deviation between the transformed line~$L_{-l,r}$ and the transformed point~$(p,f)$ in the dual space because $r-(p\cdot l+f) = r-p\cdot l -f = (-l\cdot p + r)-f$.
	Hence, it is equivalent to search for a line $L_{p,f}$ with $p,f\in \RR_{\geq 0}$ minimizing the sum of absolute deviations from the points $(l_d,r_d)$ for $d\in D$ in the primal space or to search for a point $(p,f)\in \RR_{\geq 0}\times \RR_{\geq 0}$ (i.e., in the first quadrant) that minimizes the sum of absolute deviations from the lines $L_{-l_d,r_d}$ for $d\in D$.
	The feasible space is divided into polyhedra (cells) by the lines $L_{-l_d,r_d}$ for $d\in D$ (see \Cref{fig:distance-dual-space-example-2}).
	We can determine an optimal solution to the overall problem by searching for an optimal solution in each cell and identifying the one with the best optimal objective function value.
	In each cell, the sign of $r_d-p\cdot l_d-f$ does not change for all $d\in D$ because no line is crossed, which means that the objective function can be written without the absolute value in each cell.
	Hence, in each cell, the problem is feasible, the objective function is linear and the optimal objective function value is finite.
	By the fundamental theorem of linear programming, in each cell, there is an optimal solution that is an extreme point of this cell.
	This is either the intersection of two lines, of a line with an axis, or the origin.
	Let $x^{\ast}=(p^{\ast},f^{\ast})$ be the best of all the optimal solutions of all cells.
	Interpreting the solution for the affine distance tariff design problem, this means that one of the following holds for the solution $x^{\ast}$ with price per kilometer~$p^{\ast}$ and base amount~$f^{\ast}$:
	\begin{itemize}
		\item two reference prices $r_{d_1}, r_{d_2}$, $d_1,d_2\in D$, are met exactly (if $x^\ast$ is the intersection $T((l_{d_1},r_{d_1}))\cap T((l_{{d_2}},r_{d_2}))$ of two lines in the dual space, in particular $l_{d_1}\neq l_{d_2}$),
		\item one reference price $r_d$, $d\in D$, is met exactly and $p^{\ast}=0$ (if $x^{\ast}$ is the intersection of $T((l_d,r_d))$ with the $f$-axis in the dual space),
		\item one reference price $r_d$, $d\in D$, is met exactly and $f^{\ast}=0$ (if $x^{\ast}$ is the intersection of $T((l_d,r_d))$ with the $p$-axis in the dual space),
		\item $p^{\ast}=0$ and $f^{\ast}=0$ (if $x^{\ast}=(0,0)$ is the origin in the dual space). \qedhere
	\end{itemize}
\end{proof}

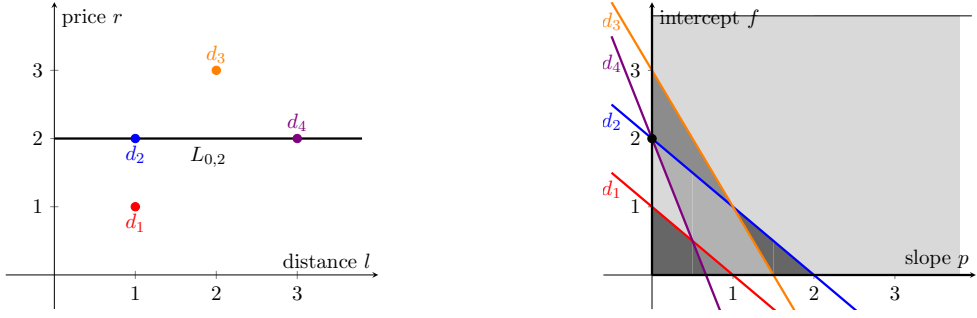
\begin{figure}
	\centering
	\begin{subfigure}[t]{0.48\textwidth}
		\centering
		\resizebox{5cm}{!}{
		\begin{tikzpicture}
			\begin{axis}[
				scale=1,
				domain=0:3.8,
				xmin=-0.6, xmax=4,
				ymin=-0.5, ymax=4,
				xtick ={0,1,2,3},
				ytick ={0,1,2,3},
				xlabel={distance $l$},
				ylabel={price $r$},
				samples=2,
				axis y line=center,
				axis x line=middle,
				]
				
				\addplot[black,mark=none,very thick] {2} node[pos=0.5, below] { $L_{0,2}$ };
				
				\addplot[red,only marks,mark=*,mark options={scale=1.1},text mark as node=true] coordinates {(1,1)} node[below] {$d_1$};
				\addplot[blue,only marks,mark=*,mark options={scale=1.1},text mark as node=true] coordinates {(1,2)} node[below] {$d_2$};
				\addplot[orange,only marks,mark=*,mark options={scale=1.1},text mark as node=true] coordinates {(2,3)} node[above] {$d_3$};
				\addplot[violet,only marks,mark=*,mark options={scale=1.1},text mark as node=true] coordinates {(3,2)} node[above] {$d_4$};
			\end{axis}
		\end{tikzpicture}}
		\caption{Primal space: points $(l_{d},r_{d})$ for four OD pairs $d\in D=\{d_1,\ldots,d_4\}$ and an optimal line $L_{0,2}$}
		\label{fig:distance-dual-space-example-1}
	\end{subfigure}
	\quad
	\begin{subfigure}[t]{0.48\textwidth}
		\centering
		\resizebox{5cm}{!}{
		\begin{tikzpicture}
			\begin{axis}[
				scale=1,
				domain=-0.5:3,
				xmin=-0.6, xmax=4,
				ymin=-0.5, ymax=4,
				xtick ={0,1,2,3},
				ytick ={0,1,2,3},
				xlabel={slope $p$},
				ylabel={intercept $f$},
				samples=2,
				axis y line=center,
				axis x line=middle,
				]
				
				\addplot[name path=d1, red,mark=none,very thick] {-1*x+1} node[pos=0, below] { $d_1$ };
				\addplot[name path=d2,blue,mark=none,very thick] {-1*x+2} node[pos=0, below] { $d_2$ };
				\addplot[name path=d3,orange,mark=none,very thick] {-2*x+3} node[pos=0, yshift={-0.35cm}] { $d_3$ };
				\addplot[name path=d4,violet,mark=none,very thick] {-3*x+2} node[pos=0, yshift={-0.5cm}] { $d_4$ };
				\addplot[name path=y,black,very thick] coordinates {(0,0)(0,3.95)};
				\addplot[name path=x,black,very thick] coordinates {(0,0)(3.95,0)};
				\addplot[name path=top,opacity=0] coordinates {(0,3.8)(3.95,3.8)};
				
				\addplot[fill=black!60!white] fill between[of=d1 and x,soft clip={domain=0:0.5}];
				\addplot[fill=black!60!white] fill between[of=d4 and x,soft clip={domain=.5:2/3}];
				
				\addplot[fill=black!15!white] fill between[of=d4 and d1,soft clip={domain=0:0.5}];
				
				\addplot[fill=black!45!white] fill between[of=d1 and d4,soft clip={domain=0.5:2/3}];
				\addplot[fill=black!45!white] fill between[of=d1 and x,soft clip={domain=2/3:1}];
				
				\addplot[fill=black!30!white] fill between[of=d2 and d4,soft clip={domain=0:0.5}];
				\addplot[fill=black!30!white] fill between[of=d2 and d1,soft clip={domain=0.5:1}];
				\addplot[fill=black!30!white] fill between[of=d3 and x,soft clip={domain=1:1.5}];
				
				\addplot[fill=black!45!white] fill between[of=d2 and d3,soft clip={domain=0:1}];
				
				\addplot[fill=black!60!white] fill between[of=d2 and d3,soft clip={domain=1:1.5}];
				\addplot[fill=black!60!white] fill between[of=d2 and x,soft clip={domain=1.5:2}];
				
				\addplot[fill=black!15!white] fill between[of=d3 and top,soft clip={domain=0:1}];
				\addplot[fill=black!15!white] fill between[of=d2 and top,soft clip={domain=1:2}];
				\addplot[fill=black!15!white] fill between[of=x and top,soft clip={domain=2:3.8}];
		
				\addplot[black,only marks,mark=*,mark options={scale=1.1},text mark as node=true] coordinates {(0,2)};
			\end{axis}
		\end{tikzpicture}}
		\caption{Dual space: The four points $(l_d,r_d)$, $d\in D$ are transformed to four lines $L_{-l_d,r_d}$. The point $(0,2)$ is the optimal solution and can be transformed back to the optimal line $L_{0,2}$ in the primal space.}
		\label{fig:distance-dual-space-example-2}
	\end{subfigure}
	\caption{Example with four OD pairs with one passenger each illustrating the proof of \Cref{prop:distance-finite-candidate-set}. The optimal solution is $p^{\ast}=0$, $f^{\ast}=2$. This shows that the three options stated in \Cref{prop:distance-finite-candidate-set} are not mutually exclusive.}
	\label{fig:distance-dual-space-example}
\end{figure}

We remark that, if an optimal solution to the restricted problem ($p,f\in \RR_{\geq 0}$) is also an optimal solution to the unrestricted problem ($p,f\in \RR$), then 
\begin{equation*}
	\sum_{d\in D:\, r_d < p\cdot l_d+f} t_d \leq \frac{\sum_{d\in D} t_d}{2} \quad \text{and} \quad \sum_{d\in D:\, r_d > p\cdot l_d +f} t_d \leq \frac{\sum_{d\in D} t_d}{2}
\end{equation*}
by \cite{Schoebel1998, Schoebel1999}, similar to the weighted median (see \Cref{definition:weighted-median}).
This means that half of the passengers pay at most as much as their reference price and half of the passengers pay at least as much in this case.

The affine distance tariff design problem can be solved by means of the finite dominating set derived in \Cref{prop:distance-finite-candidate-set}.
It can also be formulated as an LP as follows:
\begin{equation}
	\label{lp:distance}
	\begin{alignedat}{3}
		&\min_{f,p,y_d} & \sum_{d\in D} t_d y_d \quad& &&\\
		&\text{ s.t. } 	& r_d-p\cdot l_d -f&\leq y_d &\quad& \textup{for all } d\in D\\
		&				& p\cdot l_d +f -r_d &\leq y_d && \textup{for all } d\in D\\
		&				& p,f		&\geq 0	  &&\\
		&				& p,f,y_d	&\in \RR	  && \textup{for all } d\in D.
	\end{alignedat}
\end{equation}
Using this LP and a result of \cite{Zemel1984}, we can even show that the problem can be solved in linear time $\cO(\vert D \vert)$:

\begin{theorem}
	\label{prop:distance-linear}
	The affine distance tariff design problem can be solved in linear time $\cO(\vert D \vert)$.
\end{theorem}
\begin{proof}
	In \cite{Zemel1984}, the author presents an algorithm that solves the dual of the $s$-dimensional \MultipleChoiceLinearProgrammingProblem{} ($s$MCLPP) in linear time with respect to the number of constraints.
	We now show that the distance tariff design problem is of the required form for the algorithm and can hence be solved in linear time $\cO(\vert D \vert)$.
	For that, we consider $2$MCLPP (i.e., $s=2$):
	\begin{maxi*}
		{x_j}{\sum_{j\in N} c_j x_j}{}{} 
		\addConstraint{\sum_{j\in N} a^i_j x_j}{= a_0^i}{\quad\textup{for all } i \in \{1,2\}}
		\addConstraint{\sum_{j\in J_k} b_j x_j}{= b_0^k}{\quad\textup{for all } k\in \{1,\ldots,r\}}
		\addConstraint{x_j}{\geq 0}{\quad\textup{for all } j\in N}
		\addConstraint{x_j}{\in \RR}{\quad\textup{for all } j\in N.}
	\end{maxi*}
	with $r\in \NN_{\geq 1}$, $N=J_0 \cup J_1\cup \ldots \cup J_r$, the sets $J_k$, $k\in \{0,\ldots,r\}$ are pairwise disjoint and $0\notin N$. 
	Dualizing the linear program yields
	\begin{mini*}
		{w_i,v_k}{a_0^1 w_1 + a_0^2 w_2 + \sum_{k=1}^{r} b_0^k v_k}{\label{lp:zemel}}{} 
		\addConstraint{a_j^1w_1 + a_j^2 w_2}{\geq c_j}{\quad\textup{for all } j \in J_0}
		\addConstraint{a_j^1w_1 + a_j^2 w_2 + b_j v_k}{\geq c_j}{\quad\textup{for all } j \in J_k, k\in \{1,\ldots, r\}}
		\addConstraint{w_i, v_k}{\in \RR}{\quad\textup{for all } i\in \{1,2\}, k\in \{1,\ldots,r\}.}
	\end{mini*}
	We consider the special case with $J_0 = \{1,2\}$, $a_0^1 = a_0^2 =0$, $a_1^1 = a_2^2 =1$, $a_2^1 = a_1^2 = 0$, $c_1=c_2=0$, $b_j=1$ for all $j\in J_k$, $k\in \{1,\ldots,r\}$,
	which yields
	\begin{mini*}
		{w_i,v_k}{\sum_{k=1}^{r} b_0^k v_k}{\label{lp:zemel}}{} 
		\addConstraint{v_k}{\geq c_j -a_j^1 w_1 - a_j^2 w_2}{\quad\textup{for all } j \in J_k}
		\addConstraint{w_1,w_{2}}{\geq 0}{}
		\addConstraint{w_i, v_k}{\in \RR}{\quad\textup{for all } i\in \{1,2\}, k\in \{1,\ldots,r\}.}
	\end{mini*}
	By setting the sets $J_k, k\in \{1,\ldots,r\}$ and the coefficients appropriately, we obtain LP~\eqref{lp:distance}.
	Hence, by \cite{Zemel1984}, LP~\eqref{lp:distance} can be solved in linear time in the number of constraints, i.e., in $\cO(\vert D \vert)$.
\end{proof}
  
In \cite[Thm.~15, Thm.~17]{Schoebel2020}, it is shown that an affine network distance tariff always satisfies the no-elongation property and the no-stopover property, whereas an affine beeline distance tariff only satisfies the no-stopover property.
The no-elongation property is intrinsically not satisfied and can only be mitigated by using a check-in/check-out system.

\section{Zone Tariffs}
\label{sec:zone-based}

Zone-based fare structures combine the properties of flat and distance-based fare structures.
The region of the PTN is divided into tariff zones.
In a \emph{counting-zones pricing}, the price of a path depends on the number of zones traversed by the path.
This means that within a zone (as well as for each fixed number of zones) a flat tariff is applied while on a general path the distance is approximated by the number of traversed zones.
We distinguish between two types of zone tariffs: with multiple counting and with single counting.
In the multiple counting case, we count a zone each time that it is entered, as it is for example done in Boston (USA) by MBTA for commuter rail.
In the single counting case, each zone is counted at most once, which is for example used in Greater Copenhagen (Denmark) by DOT and by many German transport associations, e.g., VRN and saarVV.

Let a PTN $(V,E)$ be given.
Formally, we regard the tariff zones as a zone partition $\cZ=\{Z_1,\ldots,Z_N\}$ of the set of stations~$V\!$, i.e., $V=\bigcup_{i\in\{1,\ldots,N\}} Z_i$ and the $Z_i$ are pairwise disjoint. 
We define a \emph{zone function} $\sigma$, which counts the number of traversed zones on any path $W\in \cW$ with its set of nodes $V(W)$ and its set of edges $E(W)$.
It is different for the two ways of counting.

\textbf{Multiple Counting:}
	Let $e=\{v_1,v_2\}\in E$ be an edge.
	We define the zone border weight
	\[b(e)=b(v_1,v_2) \defeq \begin{cases} 	
		0 &\text{if } v_1 \text{ and } v_2 \text{ are in the same zone},\\
		1 & \text{otherwise.}\end{cases}\]
	From that, we derive for a path $W \in \cW$ the \emph{zone function}
	\begin{equation*}	
		\sigma(W) \defeq 1 + b(W), \mbox{ where } b(W) \defeq \sum_{e\in E(W)} b(e). \label{eq:z}
	\end{equation*}
\textbf{Single Counting:}
For every path $W \in \cW$, the \emph{zone function} that counts the number of \emph{different} zones that are traversed is defined as
\begin{equation*}\label{eq:zbar}
	\sigma(W) := \left\vert \{Z\in \cZ: V(W) \cap Z \neq \emptyset\} \right\vert.
\end{equation*}

\begin{definition}[Zone tariff \cite{Schoebel2022}]
	Let a PTN be given, and let~$\cW$ be the set of all paths in the PTN.
 	A fare structure $\pi$ is a \emph{zone tariff with multiple/single counting} w.r.t.\ a zone partition $\cZ$ and a price function $P \colon \NN_{\geq 1} \rightarrow \RR_{\geq 0}$ if $\pi(W)=P(\sigma(W))$ for all $W \in \cW$.
\end{definition}

While zone tariffs with a counting-zones pricing presented here are very popular, note that many additional specifications of zone-based fare structures exist, for example: special metropolitan zones, overlapping zones, empty zones, cumulative or maximum pricing with individual prices for the zones.

\begin{definition}[Zone tariff design problem]
	\label{definition:zone-tariff-design-problem}
	Given are a PTN $(V,E)$, a set of OD pairs $D$, reference prices $r_d\in \RR_{\geq 0}$, numbers of passengers $t_d \in \NN_{\geq 1}$ and paths $W_d\in \cW$ for all $d\in D$ and an upper bound $N\in \NN_{\geq 1}$ on the number of zones.
	The \emph{zone tariff design problem} aims to determine a zone partition $\cZ$ with at most $N$ zones and a price function $P\colon \NN_{\geq 1} \to \RR_{\geq 0}$ such that $\sum_{d\in D} t_d \vert r_d - P(\sigma(W_d))\vert$ is minimized. 
	Note that $\sigma$ depends on the zone partition~$\cZ$.
\end{definition}

\begin{remark}
	\label{remark:zone-price-list-k}
	It suffices to consider $N\in \{1,\ldots, \vert V\vert\}$ because the zones form a zone partition and it is not possible to have more than $\vert V\vert$ sets in the partition.
	Similarly, it is enough to determine the price function $P$ for input values up to $K$, where $K$ is the maximum number of nodes of a path $W_d$ for $d\in D$.
	Because $P(k)$ has no influence on the zone tariff design problem for all $k\geq K$, we simply set $P(k) = P(K)$ for all $k\geq K$.
	Therefore, a price function $P$ can be represented by a price list $(p_1,\ldots,p_K)$.
\end{remark}

In addition to distinguishing between multiple and single counting, we also distinguish two versions concerning the zones.
Given a PTN $G=(V,E)$, we say that a zone $Z\subseteq V$ is connected if its induced subgraph~$G[Z]$ is connected.
While there are no requirements in the problem formulation in \Cref{definition:zone-tariff-design-problem}, which allows \emph{arbitrary zones}, we also consider that \emph{connected zones} are demanded.
This means that each zone $Z\in\cZ$ needs to be connected.
We obtain the four variants shown in \Cref{tab:four-variants}.

\begin{table}[tbp]
	\begin{center}
		\begin{tabular}{lccc}
			\toprule
			&& multiple counting & single counting\\
			&& (M)	&	(S)\\
			\midrule
			arbitrary zones &(A)& \ZDMA{} & \ZDSA{} \\
			\addlinespace
			connected zones &(C)& \ZDMC{} & \ZDSC{} \\
			\bottomrule
		\end{tabular}
	\end{center}
	\caption{Variants of the zone tariff design problem.}
	\label{tab:four-variants}
\end{table}

Further, we consider the zone tariff design problem with and without the requirement that the resulting fare structure satisfies the no-elongation property or the no-stopover property.
The sufficient conditions for the no-elongation property and the no-stopover property in \Cref{prop:zone-properties} are independent of the PTN and the zone partition $\cZ$ and only depend on the price function~$P$.

\begin{theorem}[\cite{Urban2020,Schoebel2022}]
	\label{prop:zone-properties}
	Let a PTN, a zone partition $\cZ$ and price function $P$ be given.
	\begin{enumerate}		
		\item The zone tariff (with multiple or single counting) w.r.t.\ $\cZ$ and $P$ satisfies the no-elongation property if $P$ is monotonically increasing.
		\item The zone tariff with multiple counting w.r.t.\ $\cZ$ and~$P$ satisfies the no-stopover property if
		\begin{equation*}
			P(k) \leq P(i) + P(k-i+1) \mbox{ for all } k \in \NN_{\geq 1},\, i \in \{1,\ldots, k\}.
		\end{equation*}
		\item The zone tariff with single counting w.r.t.\ $\cZ$ and $P$ satisfies the no-stopover property if
		\begin{equation*}
			P(k)\leq P(i_1) + P(i_2) \mbox{ for all } k\in \NN_{\geq 1}, i_1, i_2 \in \{1,\ldots, k\} \mbox{ with } i_1+i_2 \geq k+1.
		\end{equation*}
	\end{enumerate}
\end{theorem}

\begin{remark}
	\label{remark:zone-properties}
	It has been shown that an equivalence is obtained in \Cref{prop:zone-properties} if the PTN and the zone partition are not fixed \cite{Urban2020,Schoebel2022}.
	The sufficient conditions of \Cref{prop:zone-properties} are however not necessary for a specific PTN with a fixed zone tariff because not all numbers of traversed zones are attained by paths in the PTN: For example, a zone tariff with only one zone and a price function with $P(1)=1$, $P(2)=3$ and $P(3)=2$ satisfies the no-elongation property although the prices are not increasing.
	Nevertheless, these sufficient conditions are used in MILPs and algorithms to ensure the no-elongation property and/or the no-stopover property.
\end{remark}

Subproblems of the zone tariff design problem are the \emph{zone-partition subproblem} and the \emph{price-setting subproblem}.
In the zone-partition subproblem, we assume the price function to be given and only optimize the zone partition.
Conversely, the zone partition is given in the price-setting subproblem and the prices are optimized.

\subsection{General Properties and Relations Between the Different Problem Variants}

In this section, we observe some general properties of the zone tariff design problem and explore the relations between the different problem variants and the behavior of the objective function values.
We start by stating that the definitions of zone functions in the case of single and multiple counting coincide under some circumstances.

\begin{lemma}
	\label{prop:multiple-equal-single}
	Let a PTN with a zone partition $\cZ$ be given.
	Both definitions of a zone function $\sigma$ for multiple and single counting coincide for paths $W\in\cW$ that do not traverse a zone several times.
	In particular, this is the case when $W=(v_1,v_2)$ for some edge $\{v_1,v_2\}\in E$.
\end{lemma}
\begin{proof}
	Let a path $W \in \cW$ that does not traverse a zone several times be given.
	In this case, every time that $b(e)=1$ for $e\in E(W)$, a new zone is entered that has not been traversed before.
	Hence, $1 + \sum_{e\in E(W)} b(e)$ is the number of different zones on the path $W$, which is equal to  $\left\vert \{Z\in \cZ: V(W) \cap Z \neq \emptyset\} \right\vert$.
\end{proof}

Next, we prove that an upper bound equal to the maximum reference price $\bar{r}\defeq \max\{r_d: d\in D\}$ is valid for the price function of an optimal solution.
	
\begin{lemma} \label{prop:bounded-prices}	
	For all optimal solution $\cZ, P$ to \ZDXY{} with $\textup{X} \in \{\textup{M, S}\}$ and  $\textup{Y} \in \{\textup{A,C}\}$ with/without requiring the no-elongation property and with/without requiring the no-stopover property, it holds that $P(k) \leq \bar{r} \defeq \max\{r_d: d\in D\}$ for all $k\in \NN_{\geq 1}$ for which there is an OD pair $d\in D$ with $\sigma(W_d)=k$.
	
	In particular, there is always an optimal solution $\cZ, P$ with $P(k)\leq \bar{r}$ for all $k\in \NN_{\geq 1}$.
\end{lemma}
\begin{proof}
	Let $\cZ, P$ be an optimal solution to \ZDXY{} with $\textup{X} \in \{\textup{M, S}\}$ and  ${\textup{Y} \in \{\textup{A,C}\}}$.
	Assume that there is some $k\in \NN_{\geq 1}$ with $P(k)>\bar{r}$ and $D_k\neq \emptyset$.
	We define a new price function $P'\colon \NN_{\geq 1} \to \RR_{\geq 0}$ by 
	\[P'(k)\defeq \begin{cases}
		P(k) &\text{if } P(k)\leq \bar{r},\\
		\bar{r} & \text{if } P(k)>\bar{r}
	\end{cases}\]
	for all $k\in \NN_{\geq 1}$.
	In order to prove the claim, we show that the zone tariff $\pi'$ w.r.t.\ $\cZ, P'$ has the same properties and yields a better objective function value than the zone tariff~$\pi$ w.r.t.\ $\cZ, P$, which leads to a contradiction.
	\begin{itemize}
		\item First, we consider an OD pairs $d\in D$ with $P(\sigma(W_d))>\bar{r}$.
		By assumption such an OD pair exists.
		Because ${P(\sigma(W_d))>\bar{r}\geq r_d}$ and $P'(\sigma(W_d))=\bar{r}$, we have that ${\vert P(\sigma(W_d)) - r_d \vert > \vert P'(\sigma(W_d)) - r_d \vert}$.
		Second, for all OD pairs $d\in D$ with ${P(\sigma(W_d))\leq \bar{r}}$, nothing is changed.
		Hence, replacing $P$ with $P'$ decreases the objective function value.
		\item In case the no-elongation property is required, this means by \Cref{definition:properties} that for all paths $W_1=(v_1,\ldots,v_n), W_2=(v_1,\ldots,v_{n-1})\in \cW$, $n\geq 2$, with $k_1\defeq \sigma(W_1)$ and $k_2\defeq \sigma(W_2)$, it holds that $P(k_2)\leq P(k_1)$.
		If $P(k_1)>\bar{r}$, then ${P'(k_2)\leq \bar{r} =P'(k_2)}$.
		If $P(k_1)\leq \bar{r}$, $P'(k_2)= P(k_2)\leq P(k_1)=P'(k_1)$.
		Hence, $\pi'$ satisfies the no-elongation property.
		\item In case the no-stopover property is required, this means by \Cref{definition:properties} that for all paths $W=(v_1,\ldots,v_n)\in \cW$ and subpaths $W_1=(v_1,\ldots,v_{i})$, $W_2=(v_i,\ldots,v_n)\in \cW$ with $n\geq 3$, $i\in\{2,\ldots,n-1\}$ and with $k\defeq \sigma(W)$, $k_1\defeq \sigma(W_1)$ and $k_2\defeq\sigma(W_2)$, it holds that ${P(k)\leq P(k_1)+P(k_2)}$.
		If $P(k_1)>\bar{r}$ or $P(k_2)>\bar{r}$, then $P'(k)\leq \bar{r}\leq P'(k_1) + P'(k_2)$.
		If $P(k_1)\leq \bar{r}$ and $P(k_2)\leq \bar{r}$, then $P'(k)\leq P(k) \leq P(k_1)+P(k_2)=P'(k_1)+P'(k_2)$.
		Hence, $\pi'$ satisfies the no-stopover property. \qedhere
	\end{itemize}
\end{proof}

The optimal objective function value of each problem \ZDXY{} with $\textup{X}\in \{\textup{M,S}\}$ and $\textup{Y}\in \{\textup{A,C}\}$ with/without requiring the no-elongation property and with/without requiring the no-stopover property is monotonically decreasing with increasing $N$ because increasing the upper bound on the number of zones extends the feasible domain.

Note that this does not hold if the zone partition needs to consist of \emph{exactly} $N$ non-empty sets as the following example shows:
Consider the linear graph depicted in \Cref{fig:ex-exact-n-zones-0}, where the OD pairs with their reference prices and paths are marked in orange.
Every OD pair has one passenger, i.e., $t_d=1$.
With two zones $\cZ=\{\{1,2\},\{3\}\}$ (\Cref{fig:ex-exact-n-zones-1}) the objective function value is 0, but for three zones the only choice of zones is $\cZ=\{\{1\},\{2\},\{3\}\}$ (\Cref{fig:ex-exact-n-zones-2}) with an objective function value of 1.
However, here, $N$ is an upper bound on the number of zones, which need not be met with equality.
\begin{figure}[tbp]
	\centering
	\begin{subfigure}[t]{0.32\textwidth}
		\centering
		\begin{tikzpicture}[thick, main/.style = {draw, circle,inner sep=1pt,minimum size=0.7cm},minor/.style = {circle,inner sep=0pt,minimum size=0cm, font=\scriptsize, fill=white}]
			\node[main] (v1) at (0,0) {$v_1$};
			\node[main] (v2) at (1.5,0) {$v_2$};
			\node[main] (v3) at (3,0) {$v_3$};
			
			\draw (v1)--(v2);
			\draw (v2)--(v3);
			
			\draw[orange, bend angle=20, bend left] (v1) to node[midway,minor]{1} (v2) ;
			\draw[orange, bend angle=30, bend left] (v1) to node[midway,minor]{2} (v3) ;
			\draw[orange, bend angle=20, bend left] (v2) to node[midway,minor]{2} (v3) ;
		\end{tikzpicture}
		\caption{PTN with OD data (orange)}
		\label{fig:ex-exact-n-zones-0}
	\end{subfigure}
	\begin{subfigure}[t]{0.32\textwidth}
		\centering
		\begin{tikzpicture}[thick, main/.style = {draw, circle,inner sep=1pt,minimum size=0.7cm}]
			\node[main,fill=red!30] (v1) at (0,0) {$v_1$};
			\node[main,fill=red!30] (v2) at (1.5,0) {$v_2$};
			\node[main,fill=blue!30] (v3) at (3,0) {$v_3$};
			
			\draw (v1)--(v2);
			\draw (v2)--(v3);
		\end{tikzpicture}
		\caption{Two zones}
		\label{fig:ex-exact-n-zones-1}
	\end{subfigure}
	\begin{subfigure}[t]{0.32\textwidth}
		\centering
		\begin{tikzpicture}[thick, main/.style = {draw, circle,inner sep=1pt,minimum size=0.7cm}]
			\node[main,fill=red!30] (v1) at (0,0) {$v_1$};
			\node[main,fill=blue!30] (v2) at (1.5,0) {$v_2$};
			\node[main,fill=yellow!30] (v3) at (3,0) {$v_3$};
			
			\draw (v1)--(v2);
			\draw (v2)--(v3);
		\end{tikzpicture}
		\caption{Three zones}
		\label{fig:ex-exact-n-zones-2}
	\end{subfigure}
	\caption{Instance showing that it can be better to implement fewer zones.}
	\label{fig:ex-exact-n-zones}
\end{figure}
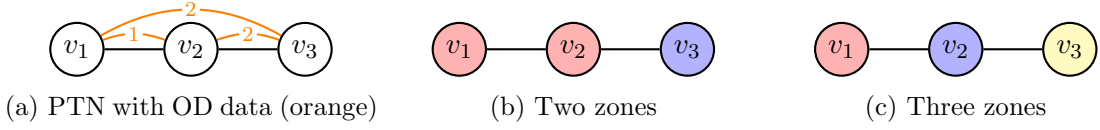

For the rest of this section, we compare the optimal objective function values of the four variants of the zone tariff design problem (see \Cref{tab:four-variants}).
Given an instance of the zone tariff design problem, let $z^{\ast}_{\textup{MA}}$, $z^{\ast}_{\textup{MC}}$, $z^{\ast}_{\textup{SA}}$ and $z^{\ast}_{\textup{SC}}$ be the respective optimal objective function values of the four variants.
The results of \Cref{prop:relations}, \Cref{prop:multiple-counting-feasible-solutions} and \Cref{ex:zdma-better-than-zdmc,ex:zdsa-better-than-zdsc,ex:zdma-better-than-zdsa-tree,ex:zdsa-better-than-zdma-tree-n,ex:zdmy-better-than-zdsy-n,ex:zdsc-better-than-zdmc-n} are summarized in \Cref{fig:relations-overview}.
All these results hold with/without requiring the no-elongation property and with/without requiring the no-stopover property (which is  not explicitly mentioned in each result for the sake of shortness).
Note that we can choose the values of the price function from the given reference prices as we will later see in \Cref{prop:fixed-zones-pricing-unconstrained}.

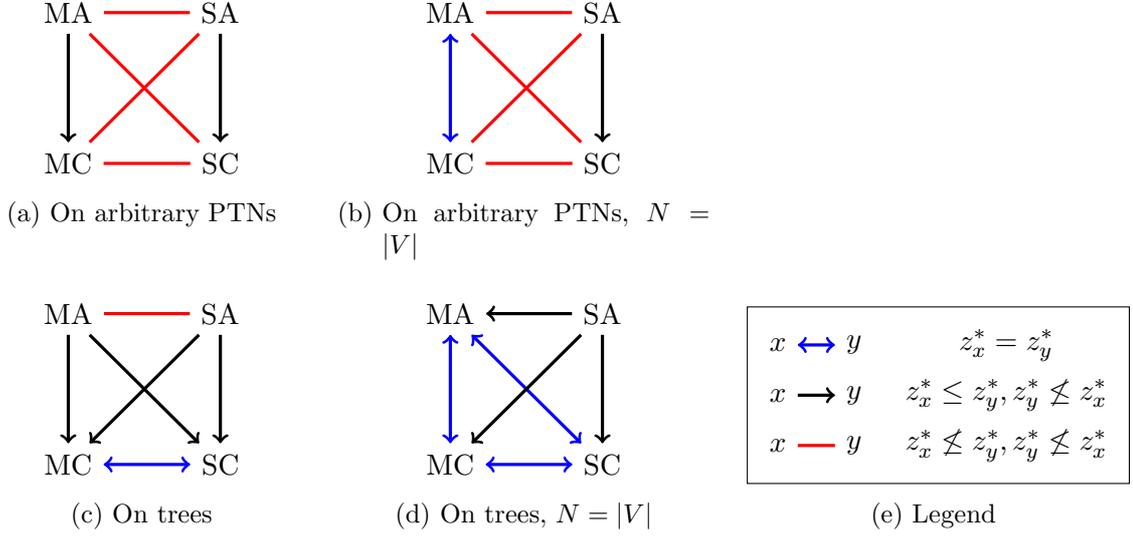
\begin{figure}[tbp]
	\begin{subfigure}[t]{0.32\textwidth}
		\centering
		\begin{tikzpicture}[very thick]
			\node[] (ma) at (0,0) {MA};
			\node[] (sa) at (2,0) {SA};
			\node[] (mc) at (0,-2) {MC};
			\node[] (sc) at (2,-2) {SC};
			
			\draw[->] (ma)--(mc);
			\draw[->] (sa)--(sc);
			\draw[red] (ma)--(sa);
			\draw[red] (ma)--(sc);
			\draw[red] (mc)--(sa);
			\draw[red] (mc)--(sc);
		\end{tikzpicture}
		\caption{On arbitrary PTNs}
	\end{subfigure}
	\begin{subfigure}[t]{0.32\textwidth}
		\centering
		\begin{tikzpicture}[very thick]
			\node[] (ma) at (0,0) {MA};
			\node[] (sa) at (2,0) {SA};
			\node[] (mc) at (0,-2) {MC};
			\node[] (sc) at (2,-2) {SC};
			
			\draw[<->,blue] (ma)--(mc);
			\draw[->] (sa)--(sc);
			\draw[red] (ma)--(sa);
			\draw[red] (ma)--(sc);
			\draw[red] (mc)--(sa);
			\draw[red] (mc)--(sc);
		\end{tikzpicture}
		\caption{On arbitrary PTNs, $N=\vert V \vert$}
	\end{subfigure}\\[1em]
	\begin{subfigure}[t]{0.32\textwidth}
		\centering
		\begin{tikzpicture}[very thick]
			\node[] (ma) at (0,0) {MA};
			\node[] (sa) at (2,0) {SA};
			\node[] (mc) at (0,-2) {MC};
			\node[] (sc) at (2,-2) {SC};
			
			\draw[->] (ma)--(mc);
			\draw[->] (sa)--(sc);
			\draw[red] (ma)--(sa);
			\draw[->] (ma)--(sc);
			\draw[->] (sa)--(mc);
			\draw[<->,blue] (mc)--(sc);
		\end{tikzpicture}
		\caption{On trees}
	\end{subfigure}
	\begin{subfigure}[t]{0.32\textwidth}
		\centering
		\begin{tikzpicture}[very thick]
			\node[] (ma) at (0,0) {MA};
			\node[] (sa) at (2,0) {SA};
			\node[] (mc) at (0,-2) {MC};
			\node[] (sc) at (2,-2) {SC};
			
			\draw[<->,blue] (ma)--(mc);
			\draw[->] (sa)--(sc);
			\draw[<-] (ma)--(sa);
			\draw[<->,blue] (ma)--(sc);
			\draw[->] (sa)--(mc);
			\draw[<->,blue] (mc)--(sc);
		\end{tikzpicture}
		\caption{On trees, $N=\vert V \vert$}
	\end{subfigure}
	\hfill
	\begin{subfigure}[t]{0.32\textwidth}
		\centering
		\begin{tikzpicture}[framed,very thick]
			\node[] (x1) at (0,0) {$x$};
			\node[] (y1) at (1,0) {$y$};
			\node[] (z1) at (3,0) {$z^{\ast}_x = z^{\ast}_y$};
			
			\node[] (x2) at (0,-2/3) {$x$};
			\node[] (y2) at (1,-2/3) {$y$};
			\node[] (z2) at (3,-2/3) {$z^{\ast}_x \leq z^{\ast}_y, z^{\ast}_y \not\leq z^{\ast}_x$};
			
			\node[] (x3) at (0,-4/3) {$x$};
			\node[] (y3) at (1,-4/3) {$y$};
			\node[] (z3) at (3,-4/3) {$z^{\ast}_x \not\leq z^{\ast}_y, z^{\ast}_y \not\leq z^{\ast}_x$};
			
			\draw[<->,blue] (x1)--(y1);
			\draw[->] (x2)--(y2);
			\draw[red] (x3)--(y3);
		\end{tikzpicture}
		\caption{Legend}
	\end{subfigure}
	\caption{Relationships between the optimal objective function values of the four problem variants of the zone tariff design problem in different cases (arbitrary PTN or tree, arbitrary $N\in \NN_{\geq 1}$ or $N=\vert V \vert$). Note that the results on trees are for the case that the paths~$W_d$ are the unique simple paths for all OD pairs~$d$ (see \Cref{prop:relations}).}
	\label{fig:relations-overview}
\end{figure}

\Cref{prop:relations} provides information about the relationship between the objective function values of the different problem variants in general as well as on trees.

\begin{lemma}
	\label{prop:relations}
	Let an instance of the zone tariff design problem be given.
	Then we have:
	\begin{enumerate}
		\item $z^{\ast}_{\textup{MA}} \leq z^{\ast}_{\textup{MC}}$ and $z^{\ast}_{\textup{SA}}\leq z^{\ast}_{\textup{SC}}$. Both may hold strictly.
		\item $z^{\ast}_{\textup{MC}}=z^{\ast}_{\textup{SC}}$ if the graph is a tree and the paths $W_d$ for $d\in D$ are the unique simple paths.
	\end{enumerate}
\end{lemma}
\begin{proof}~
	\begin{enumerate}
		\item Requiring connectedness of the zones is a restriction on the solution space. 
		\Cref{ex:zdma-better-than-zdmc,ex:zdsa-better-than-zdsc} show that the inequalities can hold strictly.
		\item Follows from \Cref{prop:multiple-equal-single}. \qedhere
	\end{enumerate}
\end{proof}

If $N=\vert V \vert$, we also obtain equality for the multiple counting cases, i.e., $z^{\ast}_{\textup{MA}} = z^{\ast}_{\textup{MC}}$, as \Cref{prop:multiple-counting-feasible-solutions} shows.
Note that it still may happen that less than $N$ zones are used, as in the example in \Cref{fig:ex-exact-n-zones}.

\begin{theorem}
	\label{prop:multiple-counting-feasible-solutions}
	Let an instance of the zone tariff design problem with upper bound $N\in \NN_{\geq 1}$ on the number of zones be given. 
	For every solution $\cZ, P$ to \ZDMA{} with upper bound $N$ on the number of zones, there is an $N'\in \NN$ with $N'\leq \vert V \vert$ and a zone partition $\cZ'$ such that $\cZ', P$ is feasible to \ZDMC{} with upper bound $N'$ on the number of zones with the same objective function value.	
	In particular, if $N=\vert V \vert$, we have $z^{\ast}_{\textup{MA}}=z^{\ast}_{\textup{MC}}$.
\end{theorem}
\begin{proof}
	Let $\cZ, P$ be a solution to \ZDMA{} regarding $N$ with $\cZ=\{Z_1,\ldots,Z_L\}$, $L\leq N$.
	We enumerate all connected components of all zones:
	For all $i\in \{1,\ldots,L\}$, let $l_i,k_i\in \NN_{\geq 1}$ with $l_i \leq k_i$, $l_1=1$ and $l_{i+1} = k_i+1$ for all $i\in \{1,\ldots,L-1\}$ such that $Z'_{l_i},\ldots, Z'_{k_i}$ denote all connected components of $G[Z_i]$ for all $i\in \{1,\ldots,L\}$.
	We set $N' \defeq k_L$ and $\cZ' \defeq \{Z'_{l_i},\ldots,Z'_{k_i}: i\in \{1,\ldots,L\}\}$.
	Then $L\leq N'\leq \vert V \vert$ and all $Z\in \cZ'$ are connected.
	For each OD pair, the number of zone borders that are crossed on its path are the same for $\cZ$ and~$\cZ'$ because no connected parts are split in the new zone partition.
	Therefore, also the objective function value remains unchanged.
\end{proof}

The following series of \Cref{ex:zdma-better-than-zdmc,ex:zdsa-better-than-zdsc,ex:zdma-better-than-zdsa-tree,ex:zdsa-better-than-zdma-tree-n,ex:zdmy-better-than-zdsy-n,ex:zdsc-better-than-zdmc-n} shows that there are no further inequalities that hold in general, on trees or with $N=\vert V\vert$.

\begin{example}[Example for $z^{\ast}_{\textup{MA}} < z^{\ast}_{\textup{MC}}$ on a tree]
	\label{ex:zdma-better-than-zdmc}
	\begin{figure}[p]
		\centering
		\begin{subfigure}[t]{0.49\textwidth}
			\centering
			\begin{tikzpicture}[thick, main/.style = {draw, circle,inner sep=1pt,minimum size=0.7cm},minor/.style = {circle,inner sep=0pt,minimum size=0cm, font=\scriptsize, fill=white}]
				\node[main] (v1) at (0,0) {$v_1$};
				\node[main] (v2) at (1.5,0) {$v_2$};
				\node[main] (v3) at (3,0) {$v_3$};
				
				\draw (v1)--(v2);
				\draw (v2)--(v3);
				
				\draw[orange, bend angle=20, bend left] (v1) to node[midway,minor]{1} (v2) ;
				\draw[orange, bend angle=30, bend left] (v1) to node[midway,minor]{2} (v3) ;
				\draw[orange, bend angle=20, bend left] (v2) to node[midway,minor]{1} (v3) ;
			\end{tikzpicture}
			\caption{PTN with OD data (orange)}
			\label{fig:ex-zdma-better-than-zdmc-0}
		\end{subfigure}
		\begin{subfigure}[t]{0.49\textwidth}
			\centering
			\begin{tikzpicture}[thick, main/.style = {draw, circle,inner sep=1pt,minimum size=0.7cm}]
				\node[main,fill=red!30] (v1) at (0,0) {$v_1$};
				\node[main,fill=blue!30] (v2) at (1.5,0) {$v_2$};
				\node[main,fill=red!30] (v3) at (3,0) {$v_3$};
				
				\draw (v1)--(v2);
				\draw (v2)--(v3);
			\end{tikzpicture}
			\caption{Optimal zones for \ZDMA{}}
			\label{fig:ex-zdma-better-than-zdmc-1}
		\end{subfigure}
		\caption{Instance for \Cref{ex:zdma-better-than-zdmc}.}
		\label{fig:ex-zdma-better-than-zdmc}
	\end{figure}
	Consider the PTN depicted in \Cref{fig:ex-zdma-better-than-zdmc-0}, which is a tree.
	The OD pairs with their reference prices and paths are marked in orange.
	Every OD pair has one passenger, i.e., $t_d=1$.
	Let $N\defeq 2$.

	Then $\cZ = \{\{v_1,v_3\}, \{v_2\}\}$ (\Cref{fig:ex-zdma-better-than-zdmc-1}) with $(p_1^{\ast}, p_2^{\ast}, p_3^{\ast})=(1,1,2)$ is an optimal solution to \ZDMA{} with objective function value $z^{\ast}_{\textup{MA}}= 0$.
	The only two structurally different connected zone partitions for \ZDMC{} are $\cZ = \{\{v_1,v_2,v_3\}\}$ and $\cZ = \{\{v_1,v_2\},\{v_3\}\}$.
	In both cases it is optimal to choose $(p_1^{\ast}, p_2^{\ast}, p_3^{\ast}) = (1,1,1)$, yielding an optimal objective function value of 1.
	Hence, $z^{\ast}_{\textup{MA}} < z^{\ast}_{\textup{MC}}$.
\end{example}

\begin{example}[Example for $z^{\ast}_{\textup{MA}} < z^{\ast}_{\textup{SA}}$ on a tree]
	\label{ex:zdma-better-than-zdsa-tree}
	\begin{figure}
		\centering
		\begin{subfigure}[t]{0.45\textwidth}
			\centering
			\begin{tikzpicture}[thick,%
				main/.style = {draw, circle,inner sep=1pt,minimum size=0.7cm},%
				minor/.style = {circle,inner sep=0pt,minimum size=0cm, font=\scriptsize, fill=white}]
				\node[main] (v1) at (0,0) {$v_1$};
				\node[main] (v2) at (1.5,0) {$v_2$};
				\node[main] (v3) at (3,0) {$v_3$};
				\node[main] (v4) at (4.5,0) {$v_4$};
				
				\draw (v1)--(v2);
				\draw (v2)--(v3);
				\draw (v3)--(v4);
				
				\draw[orange, bend angle=20, bend left] (v1) to node[midway,minor]{1} (v2) ;
				\draw[orange, bend angle=20, bend left] (v2) to node[midway,minor]{2} (v3) ;
				\draw[orange, bend angle=20, bend left] (v3) to node[midway,minor]{2} (v4) ;
				\draw[orange, bend angle=30, bend left] (v1) to node[midway,minor]{3} (v4) ;
			\end{tikzpicture}
			\caption{PTN with OD data (orange)}
			\label{fig:ex-zdma-better-than-zdsa-tree-0}
		\end{subfigure}
		\begin{subfigure}[t]{0.45\textwidth}
			\centering
			\begin{tikzpicture}[thick, main/.style = {draw, circle,inner sep=1pt,minimum size=0.7cm}]
				\node[main,fill=red!30] (v1) at (0,0) {$v_1$};
				\node[main,fill=red!30] (v2) at (1.5,0) {$v_2$};
				\node[main,fill=blue!30] (v3) at (3,0) {$v_3$};
				\node[main,fill=red!30] (v4) at (4.5,0) {$v_4$};
				
				\draw (v1)--(v2);
				\draw (v2)--(v3);
				\draw (v3)--(v4);
			\end{tikzpicture}
			\caption{Optimal zones for \ZDMA{}}
			\label{fig:ex-zdma-better-than-zdsa-tree-1}
		\end{subfigure}
		\caption{Instance for \Cref{ex:zdma-better-than-zdsa-tree}.}
		\label{fig:ex-zdma-better-than-zdsa-tree}
	\end{figure}
	Consider the PTN depicted in \Cref{fig:ex-zdma-better-than-zdsa-tree-0}, which is a tree.
	The OD pairs with their reference prices and paths are marked in orange.
	Every OD pair has one passenger, i.e., $t_d=1$.
	Let $N\defeq 2$.
	
	Then $\cZ=\{\{v_1,v_2,v_4\}, \{v_3\}\}$ (\Cref{fig:ex-zdma-better-than-zdsa-tree-1}) with $(p_1^{\ast}, p_2^{\ast}, p_3^{\ast})=(1,2,3)$ is an optimal solution to \ZDMA{} with objective function value $z^{\ast}_{\textup{MA}}=0$.
	Because it is not possible to observe three different prices with only two zones for the single counting case, we have $z^{\ast}_{\textup{MA}} < z^{\ast}_{\textup{SA}}$.
\end{example}	

\begin{example}[Example for $z^{\ast}_{\textup{SA}} < z^{\ast}_{\textup{SC}}$ on a tree with $N=\vert V \vert$]
	\label{ex:zdsa-better-than-zdsc}	
	Consider the PTN depicted in \Cref{fig:ex-zdsa-better-than-zdsc-1}, which is a tree.
	The OD pairs with their reference prices and paths are marked in orange.
	Every OD pair has one passenger, i.e., $t_d=1$.
	Let $N\defeq \vert V \vert = 5$.
	\begin{figure}[tbp]
		\centering
		\begin{subfigure}[t]{0.45\textwidth}
			\centering
			\begin{tikzpicture}[thick,%
				main/.style = {draw, circle,inner sep=1pt,minimum size=0.7cm},%
				minor/.style = {circle,inner sep=0pt,minimum size=0cm, font=\scriptsize, fill=white,text=orange}]
				\node[main] (v1) at (0,0) {$v_1$};
				\node[main] (v2) at (1.5,0) {$v_2$};
				\node[main] (v3) at (3,0) {$v_3$};
				\node[main] (v4) at (1.5,-1.5) {$v_4$};
				\node[main] (v5) at (3,-1.5) {$v_5$};
				
				\draw (v1)--(v2);
				\draw (v2)--(v3);
				\draw (v2)--(v4);
				\draw (v4)--(v5);
				
				\draw[orange, bend angle=20, bend left] (v1) to node[midway,minor]{2} (v2) ;
				\draw[orange, bend angle=20, bend left] (v2) to node[midway,minor]{1} (v3) ;
				\draw[orange, bend angle=20, bend right] (v2) to node[midway,minor]{2} (v4) ;
				\draw[orange, bend angle=20, bend right] (v4) to node[midway,minor]{2} (v5) ;
				\draw[orange] (v1.315) ..controls (1.1,-0.3) .. (v4.135) node[pos=0.3,minor]{2};
				\draw[orange] (v2.290) ..controls (1.7,-1.3).. (v5.160) 			node[pos=0.3,minor]{3};
			\end{tikzpicture}
			\caption{PTN with OD data (orange)}
			\label{fig:ex-zdsa-better-than-zdsc-1}
		\end{subfigure}
		\begin{subfigure}[t]{0.45\textwidth}
			\centering
			\begin{tikzpicture}[thick, main/.style = {draw, circle,inner sep=1pt,minimum size=0.7cm},%
				minor/.style = {circle,inner sep=0pt,minimum size=0cm, font=\scriptsize, fill=white}]
				\node[main,fill=red!30] (v1) at (0,0) {$v_1$};
				\node[main,fill=blue!30] (v2) at (1.5,0) {$v_2$};
				\node[main,fill=blue!30] (v3) at (3,0) {$v_3$};
				\node[main,fill=red!30] (v4) at (1.5,-1.5) {$v_4$};
				\node[main,fill=yellow!30] (v5) at (3,-1.5) {$v_5$};
				
				\draw (v1)--(v2);
				\draw (v2)--(v3);
				\draw (v2)--(v4);
				\draw (v4)--(v5);
			\end{tikzpicture}
			\caption{Optimal zones for \ZDSA{}}
			\label{fig:ex-zdsa-better-than-zdsc-2}
		\end{subfigure}
		\caption{Instance for \Cref{ex:zdsa-better-than-zdsc}}
		\label{fig:ex-zdsa-better-than-zdsc}
	\end{figure}
	
	Then $\cZ=\{\{v_1,v_4\},\{v_2,v_3\},\{v_5\}\}$ (\Cref{fig:ex-zdsa-better-than-zdsc-2}) with $(p_1^{\ast},p_2^{\ast},p_3^{\ast})=(1,2,3)$ is an optimal solution to \ZDSA{} with objective function value $z^{\ast}_{\textup{SA}}=0$.
	A case analysis shows that an objective function value of 0 cannot be achieved with connected zones, which yields $z^{\ast}_{\textup{SA}} < z^{\ast}_{\textup{SC}}$.
\end{example}

\begin{example}[Example for $z^{\ast}_{\textup{SA}} < z^{\ast}_{\textup{MA}}$ on a tree with $N=\vert V \vert$]
	\label{ex:zdsa-better-than-zdma-tree-n}
	Consider the PTN depicted in \Cref{fig:ex-zdsa-better-than-zdma-tree-n-0}, which is a tree.
	The OD pairs with their reference prices and paths are marked in orange.
	Every OD pair has one passenger, i.e., $t_d=1$.
	Let $N\defeq \vert V \vert =5$.
	
	Then $\cZ=\{\{v_1,v_2,v_4\},\{v_3\},\{v_5\}\}$ (\Cref{fig:ex-zdsa-better-than-zdma-tree-n-1}) with $(p_1^{\ast},p_2^{\ast},p_3^{\ast})=(1,2,3)$ is an optimal solution to \ZDSA{} with objective function value $z^{\ast}_{\textup{SA}}=0$.
	A case analysis shows that an objective function value of 0 cannot be achieved in the multiple counting case, which yields $z^{\ast}_{\textup{SA}} < z^{\ast}_{\textup{MA}}$.
	\begin{figure}[tbp]
		\centering
		\begin{subfigure}[c]{0.49\textwidth}
			\centering
			\begin{tikzpicture}[thick,%
				main/.style = {draw, circle,inner sep=1pt,minimum size=0.7cm},%
				minor/.style = {circle,inner sep=0pt,minimum size=0cm, font=\scriptsize, fill=white}]

				\node[main] (v1) at (0,0) {$v_1$};
				\node[main] (v2) at (1.5,0) {$v_2$};
				\node[main] (v3) at (3,0) {$v_3$};
				\node[main] (v4) at (4.5,0) {$v_4$};
				\node[main] (v5) at (6,0) {$v_5$};
				
				\draw (v1)--(v2);
				\draw (v2)--(v3);
				\draw (v3)--(v4);
				\draw (v4)--(v5);
				
				\draw[orange, bend angle=20, bend left] (v1) to node[midway,minor]{1} (v2) ;
				\draw[orange, bend angle=20, bend left] (v2) to node[midway,minor]{2} (v3) ;
				\draw[orange, bend angle=30, bend left] (v2) to node[midway,minor]{2} (v4) ;
				\draw[orange, bend angle=30, bend right] (v3) to node[midway,minor]{3} (v5) ;
			\end{tikzpicture}
			\caption{PTN with OD data (orange)}
			\label{fig:ex-zdsa-better-than-zdma-tree-n-0}
		\end{subfigure}%
		\begin{subfigure}[c]{0.49\textwidth}
			\centering
			\begin{tikzpicture}[thick,%
				main/.style = {draw, circle,inner sep=1pt,minimum size=0.7cm},%
				minor/.style = {circle,inner sep=0pt,minimum size=0cm, font=\scriptsize, fill=white}]
				\node[main,fill=red!30] (v1) at (0,0) {$v_1$};
				\node[main,fill=red!30] (v2) at (1.5,0) {$v_2$};
				\node[main,fill=blue!30] (v3) at (3,0) {$v_3$};
				\node[main,fill=red!30] (v4) at (4.5,0) {$v_4$};
				\node[main,fill=yellow!30] (v5) at (6,0) {$v_5$};
				
				\draw (v1)--(v2);
				\draw (v2)--(v3);
				\draw (v3)--(v4);
				\draw (v4)--(v5);
				
				\draw[draw=none, bend angle=30, bend left] (v2) to node[midway,minor,text opacity=0]{2} (v4) ;
				\draw[draw=none, bend angle=30, bend right] (v3) to node[midway,minor,text opacity=0]{3} (v5) ;
			\end{tikzpicture}
			\caption{Optimal zones for \ZDSA{}}
			\label{fig:ex-zdsa-better-than-zdma-tree-n-1}
		\end{subfigure}\\[1em]
		\caption{Instance for \Cref{ex:zdsa-better-than-zdma-tree-n}.}
		\label{fig:ex-zdsa-better-than-zdma-tree-n}
	\end{figure}
\end{example}

\begin{example}[Example for $z^{\ast}_{\textup{MY\textsubscript{1}}} < z^{\ast}_{\textup{SY\textsubscript{2}}}$ with $N=\vert V \vert$ for all $\textup{Y\textsubscript{1}}, \textup{Y\textsubscript{2}} \in \{\textup{A},\textup{C}\}$]
	\label{ex:zdmy-better-than-zdsy-n}
	
	Consider the PTN depicted in \Cref{fig:ex-zdmy-better-than-zdsy-n-0}.
	The OD pairs with their reference prices and paths are marked in orange.
	Every OD pair has one passenger, i.e., $t_d=1$.
	Let $N\defeq \vert V \vert = 5$.
	\begin{figure}
		\centering
		\begin{subfigure}[t]{0.45\textwidth}
			\centering
			\begin{tikzpicture}[thick,%
				main/.style = {draw, circle,inner sep=1pt,minimum size=0.7cm},%
				minor/.style = {circle,inner sep=0pt,minimum size=0cm, font=\scriptsize, fill=white}]
				\node[main] (v1) at (0,0) {$v_1$};
				\node[main] (v2) at (1.5,0) {$v_2$};
				\node[main] (v3) at (3,0) {$v_3$};
				\node[main] (v4) at (1.5,-1.5) {$v_4$};
				\node[main] (v5) at (3,-1.5) {$v_5$};
				
				\draw (v1)--(v2);
				\draw (v2)--(v3);
				\draw (v2)--(v4);
				\draw (v3)--(v5);
				\draw (v4)--(v5);
				
				\draw[orange, bend angle=20, bend left] (v2) to node[midway,minor]{1} (v3);
				\draw[orange, bend angle=20, bend left] (v3) to node[midway,minor]{1} (v5);
				\draw[orange, bend angle=20, bend left] (v1) to node[midway,minor]{2} (v2);
				\draw[orange, bend angle=30, bend left] (v1) to node[midway,minor]{2} (v3);
				\draw[orange] (v2.290) ..controls (1.7,-1.3).. (v5.160) 			node[pos=0.3,minor]{3};
			\end{tikzpicture}
			\caption{PTN with OD data (orange)}
			\label{fig:ex-zdmy-better-than-zdsy-n-0}
		\end{subfigure}
		\begin{subfigure}[t]{0.45\textwidth}
			\centering
			\begin{tikzpicture}[thick, main/.style = {draw, circle,inner sep=1pt,minimum size=0.7cm},%
				minor/.style = {circle,inner sep=0pt,minimum size=0cm, font=\scriptsize, fill=white}]
				\node[main,fill=red!30] (v1) at (0,0) {$v_1$};
				\node[main,fill=blue!30] (v2) at (1.5,0) {$v_2$};
				\node[main,fill=blue!30] (v3) at (3,0) {$v_3$};
				\node[main,fill=yellow!30] (v4) at (1.5,-1.5) {$v_4$};
				\node[main,fill=blue!30] (v5) at (3,-1.5) {$v_5$};
				
				\draw (v1)--(v2);
				\draw (v2)--(v3);
				\draw (v2)--(v4);
				\draw (v3)--(v5);
				\draw (v4)--(v5);
			\end{tikzpicture}
			\caption{Optimal zones for \ZDMA{} and \ZDMC{}}
			\label{fig:ex-zdmy-better-than-zdsy-n-1}
		\end{subfigure}
		\caption{Instance for \Cref{ex:zdmy-better-than-zdsy-n}}
		\label{fig:ex-zdmy-better-than-zdsy-n}
	\end{figure}
	
	Then $\cZ=\{\{v_1\},\{v_2,v_3,v_5\},\{v_4\}\}$ (\Cref{fig:ex-zdmy-better-than-zdsy-n-1}) with $(p_1^{\ast},p_2^{\ast},p_3^{\ast})=(1,2,3)$ is an optimal solution to \ZDSA{} and \ZDSC{} with objective function value $ z^{\ast}_{\textup{SA}}=z^{\ast}_{\textup{SC}}=0$.
	A case analysis shows that an objective function value of 0 cannot be achieved in the single counting case, which yields $z^{\ast}_{\textup{MA}}=z^{\ast}_{\textup{MC}}< z^{\ast}_{\textup{SA}}$ and $z^{\ast}_{\textup{MA}}=z^{\ast}_{\textup{MC}}<z^{\ast}_{\textup{SC}}$.
\end{example}

\begin{example}[Example for $z^{\ast}_{\textup{SC}} < z^{\ast}_{\textup{MC}}$ with $N=\vert V \vert$]
	\label{ex:zdsc-better-than-zdmc-n}
	
	Consider the PTN depicted in \Cref{fig:ex-zdsc-better-than-zdmc-n-0}.
	The OD pairs with their reference prices and paths are marked in orange.
	Every OD pair has one passenger, i.e., $t_d=1$.
	Let $N\defeq \vert V \vert = 5$.
	\begin{figure}[tbp]
		\centering
		\begin{subfigure}[t]{0.45\textwidth}
			\centering
			\begin{tikzpicture}[thick,%
				main/.style = {draw, circle,inner sep=1pt,minimum size=0.7cm},%
				minor/.style = {circle,inner sep=0pt,minimum size=0cm, font=\scriptsize, fill=white}]
				\node[main] (v1) at (0,0) {$v_1$};
				\node[main] (v2) at (1.5,0) {$v_2$};
				\node[main] (v3) at (3,0) {$v_3$};
				\node[main] (v4) at (1.5,-1.5) {$v_4$};
				\node[main] (v5) at (3,-1.5) {$v_5$};
				
				\draw (v1)--(v2);
				\draw (v2)--(v3);
				\draw (v2)--(v4);
				\draw (v3)--(v5);
				\draw (v4)--(v5);
				
				\draw[orange, bend angle=20, bend left] (v2) to node[midway,minor]{1} (v3);
				\draw[orange, bend angle=20, bend left] (v3) to node[midway,minor]{1} (v5);
				\draw[orange, bend angle=20, bend left] (v1) to node[midway,minor]{2} (v2);
				\draw[orange, bend angle=30, bend left] (v1) to node[midway,minor]{2} (v3);
				\draw[orange] (v2.290) ..controls (1.7,-1.3).. (v5.160) 			node[pos=0.3,minor]{2};
				\draw[orange] (v1.-20) ..controls (1.3,-0.2).. (v4.110) 			node[pos=0.3,minor]{3};
			\end{tikzpicture}
			\caption{PTN with OD data (orange)}
			\label{fig:ex-zdsc-better-than-zdmc-n-0}
		\end{subfigure}
		\begin{subfigure}[t]{0.45\textwidth}
			\centering
			\begin{tikzpicture}[thick, main/.style = {draw, circle,inner sep=1pt,minimum size=0.7cm},%
				minor/.style = {circle,inner sep=0pt,minimum size=0cm, font=\scriptsize, fill=white}]
				\node[main,fill=red!30] (v1) at (0,0) {$v_1$};
				\node[main,fill=blue!30] (v2) at (1.5,0) {$v_2$};
				\node[main,fill=blue!30] (v3) at (3,0) {$v_3$};
				\node[main,fill=yellow!30] (v4) at (1.5,-1.5) {$v_4$};
				\node[main,fill=blue!30] (v5) at (3,-1.5) {$v_5$};
				
				\draw (v1)--(v2);
				\draw (v2)--(v3);
				\draw (v2)--(v4);
				\draw (v3)--(v5);
				\draw (v4)--(v5);
			\end{tikzpicture}
			\caption{Optimal zones for \ZDSC{}}
			\label{fig:ex-zdsc-better-than-zdmc-n-1}
		\end{subfigure}
		\caption{Instance for \Cref{ex:zdsc-better-than-zdmc-n}}
		\label{fig:ex:zdsc-better-than-zdmc-n}
	\end{figure}
	
	Then $\cZ=\{\{v_1\},\{v_2,v_3,v_5\},\{v_4\}\}$ (\Cref{fig:ex-zdmy-better-than-zdsy-n-1}) with $(p_1^{\ast},p_2^{\ast},p_3^{\ast})=(1,2,3)$ is an optimal solution to \ZDSA{} and \ZDSC{} with objective function value $ z^{\ast}_{\textup{SA}}=z^{\ast}_{\textup{SC}}=0$.
	A case analysis shows that an objective function value of 0 cannot be achieved in the multiple counting case, which yields $z^{\ast}_{\textup{SA}}=z^{\ast}_{\textup{SC}}< z^{\ast}_{\textup{MA}}$ and $z^{\ast}_{\textup{SA}}=z^{\ast}_{\textup{SC}}<z^{\ast}_{\textup{MC}}$.
	
\end{example}

Because the price lists constructed in \Cref{ex:zdma-better-than-zdmc,ex:zdma-better-than-zdsa-tree,ex:zdmy-better-than-zdsy-n,ex:zdsa-better-than-zdma-tree-n,ex:zdsa-better-than-zdsc,ex:zdsc-better-than-zdmc-n} clearly satisfy the no-elongation property and the no-stopover property, all these results hold with and without requiring the no-elongation property and with and without requiring the no-stopover property.

\subsection{Complexity and Algorithms}

In \cite[Thm.~2]{Hamacher2004} it is shown that the zone tariff design problems \ZDMA{} and \ZDMC{} \emph{minimizing the maximum absolute deviation} from reference prices are NP-hard in case that $N\geq 3$ is fixed, the zone partition needs to consist of exactly $N$ (non-empty) sets and the passengers' paths minimize the number of traversed zones.
Further, for fixed passengers' paths, \cite[Thm.~1]{Otto2017} proves NP-hardness of the zone tariff design problems \ZDSA{} (even if $N=2$) and \ZDSC{} \emph{maximizing the revenue} with additional constraints that take into account a limited willingness to pay of the passengers.
Also, \cite[Thm.~4]{Otto2017} shows that setting the prices can be done in $\cO(\kappa\cdot \vert D \vert)$ in case the zones are given, where $\kappa$ is the maximum number of traversed zones.
We contribute to these results with a complexity analysis of the zone tariff design problems \ZDXY{} with $\textup{X} \in \{\textup{M, S}\}$ and  $\textup{Y} \in \{\textup{A,C}\}$ that \emph{minimize the sum of absolute deviations} from reference prices including the no-elongation property and the no-stopover property.
An overview of the results is given in \Cref{tab:complexity-results}.
\begin{table}
	\centering
	\begin{tabular}{lc@{\hskip 4pt}c@{\hskip 4pt}c@{\hskip 4pt}cllll@{~}l}
		\toprule
		&\multicolumn{4}{c}{ZD-}& no- & no- & complexity  & \multicolumn{2}{l}{reference} \\
		&MA & MC & SA & SC &elong.&stop.&&&\\
		\midrule
		general &x & &x&  & w/wo & w/wo & NP-hard &  Thm.&\ref{prop:NP-arbitrary-zones}\\
		\addlinespace
		zone-partition & x & & x & & w/wo & w/wo & NP-hard &  Cor.&\ref{prop:NP-arbitrary-zones-fixed-price}\\
		\addlinespace
		general & & x & & x & w/wo & w/wo & NP-hard & Thm.&\ref{prop:NP-connected-zones}\\
		\addlinespace
		zone-partition & & x & & x & w/wo & w/wo & NP-hard & Cor.&\ref{prop:NP-connected-zones-fixed-price}\\
		\addlinespace
		price-setting & x & x & x & x & wo & wo & $\cO(\vert D \vert)$ &Thm. &\ref{prop:fixed-zones-pricing-unconstrained}\\
		\addlinespace
		price-setting & x & x & x & x & w & wo & $\cO(\kappa \cdot \vert D \vert)$ &Thm.&\ref{prop:zone-no-elong-alg-correctness}\\
		\addlinespace
		price-setting & x & x & x & x & w/wo & w & polynomial & LP & \eqref{lp:price-setting-poly}\\
		\bottomrule
	\end{tabular}
	\caption{Overview of the complexity results for the variants of the zone tariff design problem. Abbreviations: w $=$ with, wo $=$ without.}
	\label{tab:complexity-results}
\end{table}

In order to analyze the complexity, we consider the decision version of the zone tariff design problem.
The problem changes such that we have an additional input parameter $J \in \RR_{\geq 0}$ and search for a zone partition and a price function such that $\sum_{d\in D} t_d \vert r_d - P(\sigma(W_d))\vert \leq J$.

Note that all four variants \ZDMA{}, \ZDMC{}, \ZDSA{}, \ZDSC{} are in NP.

First, we show that the zone tariff design problem with arbitrary zones is NP-complete.

\begin{theorem}
	\label{prop:NP-arbitrary-zones}
	The problems \ZDMA{} and \ZDSA{} are NP-complete
	\begin{itemize}
		\item with/without requiring the no-elongation property,
		\item with/without requiring the no-stopover property,
		\item even if $N=2$.
	\end{itemize}
\end{theorem}
\begin{proof}
	We use a reduction from \BipartiteSubgraph{}~\cite[Problem GT25]{Garey1979}, which is defined as follows:\\
	\textbf{Instance:} graph $G=(V,E)$, positive integer $Q'\in \NN_{\geq 1}$ with $Q'\leq \vert E \vert$.\\
	\textbf{Question:} Is there a subset $E'\subseteq E$ with $\vert E' \vert \geq Q'$ such that $(V,E')$ is bipartite?
	
	We first give the following equivalent formulation of \BipartiteSubgraph{}:
	Are there sets $A,B\subseteq V$ such that $A\cup B = V\!$, $A\cap B=\emptyset$ and 
	\[\interior(A,B)\defeq \{\{v,w\}\in E: v,w\in A\}\cup \{\{v,w\}\in E: v,w\in B\}\]
	satisfies $\vert \interior(A,B)\vert \leq Q \defeq \vert E \vert -Q'$?
	In this case, $(A,B)$ is the bipartition of the graph, $\interior(A,B)$ is the edge set that needs to be deleted to receive a bipartite graph regarding the bipartition $(A,B)$, and the remaining edges are those of the cut set
	\[\cut(A,B) \defeq \{\{v,w\} \in E: \{v,w\}\cap A \neq \emptyset, \{v,w\}\cap B \neq \emptyset\},\]
	which are $\vert \cut(A,B) \vert = \vert E\vert - \vert \interior(A,B)\vert \geq Q'$ many.
	
	Let an instance $G=(V,E), Q$ of \BipartiteSubgraph{} be given. 
	We construct an instance of \textup{ZD-XA}, $\textup{X}\in \{\textup{M, S}\}$.
	Let $x_1$ and $x_2$ be additional nodes not contained in $V\!$, and let $u\in V$ be chosen arbitrarily but fixed.
	We define:
	\begin{equation*}
		\begin{aligned}
			\bar{V} &\defeq V \cup \{x_1,x_2\},\\
			\bar{E} &\defeq E \cup \{ \{u,x_1\},\{u,x_2\}, \{x_1,x_2\}\},\\
			D &\defeq \{(v_1,v_2): \{v_1,v_2\}\in \bar{E}\},\\
			r_d &\defeq \begin{cases}
				2	&\textup{for all } d\in D\setminus \{(x_1,x_2)\},\\
				1	&\textup{for } d=(x_1,x_2),
			\end{cases}\\
			t_d &\defeq \begin{cases}
				1	&\textup{for all } d\in D\setminus \{(u,x_1),(u,x_2), (x_1,x_2)\},\\
				M \defeq \vert E \vert +1	&\textup{for all } d\in \{(u,x_1),(u,x_2), (x_1,x_2)\},
			\end{cases}\\
			W_d &\defeq d \text{\quad for all } d\in D,\\
			N &\defeq 2,\\
			J &\defeq Q\leq \vert E\vert.
		\end{aligned}
	\end{equation*}

	\begin{figure}
		\centering
		\begin{tikzpicture}[thick, main/.style = {draw, circle,inner sep=1pt,minimum size=.7cm}]
			\node[font=\large] (G) at (-1.2,0.5) {$G$};
			\node[main] (u) at (1,0) {$u$};
			\node[main] (x1) at (4,1) {$x_1$};
			\node[main] (x2) at (4,-1) {$x_2$};
			\draw (0,0) ellipse [x radius=2cm,y radius=1cm];
			
			\draw (u)--(x1) node[style={pos=0.65}, fill=white, sloped, font=\footnotesize] {(\color{red}{2},\color{blue}{$M$})};
			\draw (u)--(x2) node[style={pos=0.65}, fill=white, sloped, font=\footnotesize] {(\color{red}{2},\color{blue}{$M$})};
			\draw (x1)--(x2) node[midway, fill=white, font=\footnotesize] {(\color{red}{1},\color{blue}{$M$})};
			
			\draw[dashed] (u)--(0.5,-0.5);
			\draw[dashed] (u)--(0.3,0) node[left, fill=white, sloped, font=\footnotesize] {(\color{red}{2},\color{blue}{1})};
			\draw[dashed] (u)--(0.5,0.5);
		\end{tikzpicture}
		\caption{Graph construction for the proof of \Cref{prop:NP-arbitrary-zones}. The node $u$ is an arbitrary but fixed node in the graph $G$ (indicated by the ellipsoid). The additional nodes $x_1$ and $x_2$ are connected with each other and with $u$. The OD pairs correspond to the edges. The reference prices $r_d$ (red) and numbers of passengers $t_d$ (blue) per OD pair $d\in D$ are given on the edges as $(r_d,t_d)$.}
		\label{fig:NP-arbitrary-zones}
	\end{figure}
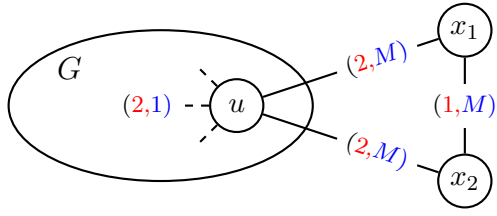

	This means that each OD pair corresponds to an edge and each edge in the new network is a path used by one passenger with reference price 2 for the edges in $E$, and by $M > \vert E \vert$ passengers with reference price 1 or 2 for the edges in $\bar{E}\setminus E$.
	For the maximum number of nodes of a path $K$, we thus have $K=2$.
	The construction is shown in \Cref{fig:NP-arbitrary-zones}.
	Solving the zone tariff design problem on this instance, we need to determine two zones $A$ and $B$ and two prices $(p_1,p_2)$ (see \Cref{remark:zone-price-list-k}).
	For the zone partition, we have the following three options for $x_1, x_2, u$:
	\begin{itemize}
		\item First, $x_1,x_2,u$ are in the same zone.
		Because $\sigma((u,x_1))= \sigma((x_1,x_2))=1$, the contribution to the objective function value is at least $M\cdot\vert 2-p_1 \vert + M \cdot \vert 1-p_1\vert \geq M = \vert E \vert +1 > Q = J$.
		Hence, this option does not yield a feasible solution.	
		
		\item Second, $x_1,u\in A, x_2\in B$ (w.l.o.g.).
		Because $\sigma((u,x_2))= \sigma((x_1,x_2))=2$, the contribution to the objective function value is at least $M\cdot\vert 2-p_2 \vert + M \cdot \vert 1-p_2\vert \geq M = \vert E \vert +1 > Q = J$.
		Hence, this option does also not yield a feasible solution.
		
		\item Third, $x_1,x_2\in A, u\in B$ (w.l.o.g.).
		The objective function value then is
		\begin{equation*}
			\begin{aligned}
				\sum_{d\in D} t_d \vert r_d - P(\sigma(W_d))\vert
				={}& M \cdot \vert 1- p_1\vert + 2\cdot M\cdot \vert 2- p_2\vert\\
				&+ \sum_{e\in \interior(A,B)}\vert 2-p_1\vert + \sum_{e\in \cut(A,B)} \vert 2-p_2 \vert,\\
			\end{aligned}
		\end{equation*}
		where it is optimal to choose $p^{\ast}_2 = 2$.
		Because $M = \vert E\vert +1 > \vert \interior(A,B)\vert$, it is optimal to choose $p^{\ast}_1=1$. 
		This yields an overall objective function value of $\vert \interior(A,B)\vert$.
	\end{itemize}
	It is hence only possible to obtain a feasible solution if $x_1$ and $x_2$ are in the same zone and $u$ is in the other zone.
	A solution to the remaining problem in this third option is a solution to \BipartiteSubgraph{} and the other way around:
	In both cases, a bipartition of the nodes~$V$ into sets~$A$ and~$B$ has to be found such that $\vert \interior(A,B) \vert \leq Q$.
	
	Note that the resulting prices satisfy the no-elongation property (due to the monotonicity of the price list) and the no-stopover property (because the price $P(k)$ is constant for $k\geq 2$).
	Hence, the construction does the same if these properties are required.
	Because each path $W_d$, $d\in D$ consists of only one edge, the number of zones with multiple counting and with single counting are the same by \Cref{prop:multiple-equal-single}.
\end{proof}

Next, we show NP-completeness of the zone tariff design problem with the requirement of connected zones.
This does not follow from the proof of \Cref{prop:NP-arbitrary-zones} because the resulting zones will in most cases not be connected.
For the NP-completeness proof, we make use of the NP-complete problem \Multicut{}, which was used in \cite{Otto2017}, as well.

\begin{theorem}
	\label{prop:NP-connected-zones}
	The problems \ZDMC{} and \ZDSC{} for connected zones are NP-complete
	\begin{itemize}
		\item with/without requiring the no-elongation property,
		\item with/without requiring the no-stopover property,
		\item even if the graph is a tree.
	\end{itemize}
\end{theorem}
\begin{proof}
	We use a reduction from \Multicut{} on a star graph with unit weights:\\
	\textbf{Instance:} graph $G=(V,E)$, edge weights $c_e\in \RR_{\geq 0}$ for all $e\in E$, a set of source-terminal pairs $\cC \subseteq V\times V\!$, non-negative integer $Q\in \NN_{\geq 1}$.\\
	\textbf{Question:} Is there a subset $\bar{E}\subseteq E$ with $\sum_{e\in \bar{E}} c_e \leq Q$ and all source-terminal pairs in $\cC$ are separated by $\bar{E}$?\\
	It was shown in \cite[Prop.~3.1, Thm.~3.1]{Garg1997} that \Multicut{} is NP-complete even on a star graph with unit weights $c_e=1$ for all $e\in E$. 
	In the following, we denote this special case by \Multicut{}.
	Because the problem is polynomially solvable if $\vert \cC \vert \leq Q$ (delete one edge per source-terminal pair in $\cC$), we assume that $Q < \vert \cC \vert$.
	Further, $Q<\vert E \vert = \vert V \vert -1$ because otherwise all edges of the tree $(V,E)$ can be deleted and the problem is trivial.
	
	Let an instance of \Multicut{} consisting of a star graph $G=(V,E)$, unit costs $c_e =1$ for all $e\in E$, source-terminal pairs $\cC$, a non-negative integer $Q$ be given. 
	We construct an instance of \textup{ZD-XC}, $\textup{X}\in \{\textup{M, S}\}$.
	Let $x_1,\ldots, x_{Q+1}$ be additional nodes not contained in $V\!$, and let $x_0\in V$ be the center of the star graph $G$.
	We define:
	\begin{equation*}
		\begin{aligned}
			\bar{V} &\defeq V \cup V_x \text{\quad with } V_x \defeq \{x_j: j\in \{1,\ldots, Q+1\}\},\\
			\bar{E} &\defeq E \cup E_x \text{\quad with } E_x\defeq \{ \{x_j,x_{j+1}\}: j\in \{0,\ldots, Q\}\},\\
			D 		&\defeq \cC \cup D_x \text{\quad with } D_x \defeq \{(x_j,x_{j+1}): j\in \{0,\ldots,Q\}\},\\
			r_d 	&\defeq \begin{cases}
				1	&\textup{for all } d\in D_x,\\
				2	&\textup{for all } d\in \cC,
			\end{cases}\\
			t_d 	&\defeq 1 \text{\quad for all } d\in D,\\
			W_d &\text{ is the unique simple path in } (\bar{V},\bar{E}) \text{ for all } d\in D,\\
			N &\defeq Q+1,\\
			J &\defeq 0.
		\end{aligned}
	\end{equation*}
	This means that each OD pair has one passenger and corresponds either to a source-terminal pair with a reference price of 2, or to a newly added edge with a reference price of 1.
	Hence, for the maximum number of nodes of a path $K$, we have $K\leq 3$.
	Because $Q<\vert \cC \vert$, this is a polynomial reduction.
	The construction is depicted in \Cref{fig:NP-connected-zones}.
	Solving the zone tariff design problem on this instance, we need to determine at most $N$ zones and a price list $(p_1,p_2,p_3)$ (see \Cref{remark:zone-price-list-k}).
	
	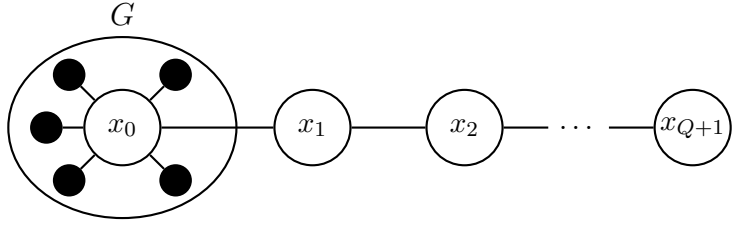
\begin{figure}
		\centering
		\begin{tikzpicture}[thick, main/.style = {draw, circle,inner sep=1pt,minimum size=1cm}, small/.style = {draw, circle, fill=black, inner sep=1pt,minimum size=.4cm}]
			\node[font=\large] (G) at (0,1.5) {$G$};
			\node[main] (v) at (0,0) {$x_0$};
			\draw (0,0) ellipse [x radius=1.5cm,y radius=1.2cm];
			
			\node[main] (x1) at (2.5,0) {$x_1$};
			\node[main] (x2) at (4.5,0) {$x_2$};
			\node[]   (u)   at (6,0) {$\dots$};
			\node[main] (x3) at (7.5,0) {$x_{Q+1}$};
			
			\node[small] (s1) at (0.7,-0.7) {};
			\node[small] (s2) at (0.7,0.7) {};
			\node[small] (s3) at (-0.7,0.7) {};
			\node[small] (s4) at (-1,0) {};
			\node[small] (s5) at (-0.7,-0.7) {};
			
			\draw (v)--(s1);
			\draw (v)--(s2);
			\draw (v)--(s3);
			\draw (v)--(s4);
			\draw (v)--(s5);
			
			\draw (v) -- (x1) -- (x2) -- (u)-- (x3);
		\end{tikzpicture}
		\caption{Graph construction for the proof of \Cref{prop:NP-connected-zones}. The star graph $G$ with center $x_0$ is indicated in the ellipsoid. It is extended by a path consisting of $x_0,\ldots, x_{Q+1}$.}
		\label{fig:NP-connected-zones}
	\end{figure}
	
	We show that there is a solution $E'$ to \Multicut{} if and only if there is a solution $\cZ,P$ to \textup{ZD-XC}, ${\textup{X}\in \{\textup{M, S}\}}$.
	
	For the first direction, let $E'\subseteq E$ be a solution to \Multicut{}.
	Deleting an edge $e\in E'$ in the star graph generates a new connected component.
	Thus $(\bar{V}, \bar{E}\setminus E')$ has $L \defeq 1+\vert E'\vert \leq 1+Q =N$ connected components.
	We define the connected components to be the zones $Z_1,\ldots, Z_L$, in particular $G[Z_i]$ is connected for all $i\in \{1,\ldots, L\}$.
	Because all pairs in $\cC$ are separated by $E'$ and $E' \cap E_x =\emptyset$, we have $\sigma(W_d)\in \{2,3\}$ for all $d\in D$, and $\sigma(W_d)=1$ for all $d\in D_x$.
	By setting $p^{\ast}_1 =1$ and $p^{\ast}_2 = p^{\ast}_3 = 2 \eqdef p^{\ast}_{2,3}$, we obtain a feasible solution to \textup{ZD-XC}, ${\textup{X}\in \{\textup{M, S}\}}$:
	\begin{equation*}
		\begin{aligned}
			\sum_{d\in D} t_d \vert r_d - P(\sigma(W_d)) \vert
			=  \sum_{d\in \cC} \vert 2- p^{\ast}_{2,3} \vert + \sum_{d\in D_x} \vert 1- p_1^{\ast} \vert
			= \sum_{d\in \cC} \vert 2 - 2 \vert + \sum_{d\in D_x} \vert 1 -1 \vert
			= 0.
		\end{aligned}
	\end{equation*}
	
	For the other direction, let a zone partition $\cZ=\{Z_1,\ldots,Z_L\}$ with $L\leq N$ and a price function~$P$ be a solution to \textup{ZD-XC}, ${\textup{X}\in \{\textup{M, S}\}}$.
	Because $J =0$, we have $P(\sigma(W_d))=r_d$ for all $d\in D$.
	We set 
	\[E' \defeq \{ \{v_1,v_2\}\in E: v_1\in Z_i, v_2\in Z_j, i\neq j\}.\]
	It holds that $\vert E' \vert \leq Q = N-1$ because $G=(V,E)$ is a star graph and $G[Z_i]$ is connected for all $i\in \{1,\ldots, L\}$: 
	Otherwise, if $\vert E' \vert>N-1$, there would be more than $N$ connected components, which would mean that at least one zone $Z_i$ would not be connected.
	Next, we show that all source-terminal pairs in $\cC$ are separated by $E'$.
	Note that it cannot happen that each of the nodes $x_1, \ldots, x_{Q+1}$ forms a singleton zone only containing that node because then at least $Q+2$ zones would be needed, which is not feasible because $Q+2 >N$.
	Hence, there is an $i\in \{1,\ldots, L\}$ and a $j\in \{0,\ldots, Q\}$ such that $\{x_j, x_{j+1} \}\subseteq Z_i$.
	Therefore, $\sigma(W_{(x_j,x_{j+1})})=1$.
	Because $J=0$, we have that $p_1 = r_{(x_j,x_{j+1})} = 1$.
	Again because $J =0$ and $r_d=2 \neq 1 = p_1$ for all $d\in \cC$, we obtain $\sigma(W_d)\in \{2,3\}$ for all $d\in \cC$ and $p_2=p_3=2$.
	Hence, no source-terminal pair is in the same zone.
	This means that for all $d\in \cC$, there exists an edge $\{v_1,v_2\}\in E(W_d)$ with $v_1$ and $v_2$ in different zones, and hence $\{v_1,v_2\} \in E'$.
	Therefore, all source-terminal pairs in $\cC$ are separated by $E'$.
	
	Note that the resulting prices satisfy the no-elongation property and the no-stopover property.
	Because the zones are connected and the paths are the unique simple paths in the tree, the number of zones with multiple counting and with single counting are the same by \Cref{prop:relations},.
\end{proof}

\paragraph*{The Zone-Partition Subproblem}
Let us now consider the zone-partition subproblem, i.e., the zone tariff design problem with a given price function that only asks for the zone partition.
For $N=1$, this is simple because there is exactly one feasible solution, namely all nodes are in the same zone.
However, the zone-partition subproblem with arbitrary zones is already NP-complete if we have $N=2$, and the zone-partition subproblem with connected zones is already NP-complete if the graph is a tree as the following corollaries show.

\begin{corollary}
	\label{prop:NP-arbitrary-zones-fixed-price}
	The zone-partition subproblem of \ZDMA{} and \ZDSA{} (with arbitrary zones) is NP-complete, even if $N=2$.
\end{corollary}
\begin{proof}
	The proof of \Cref{prop:NP-arbitrary-zones} works analogously if we set $p_1^{\ast}\defeq 1$ and $p_2^{\ast}\defeq 2$ already in the construction of the instance of the zone tariff design problem. 
\end{proof}

\begin{corollary}
	\label{prop:NP-connected-zones-fixed-price}
	The zone-partition subproblem of \ZDMC{} and \ZDSC{} (with connected zones) is NP-complete, even if the graph is a tree.
\end{corollary}
\begin{proof}
	The proof of \Cref{prop:NP-connected-zones} works analogously if we set $p_1^{\ast}\defeq 1$ and $p_2^{\ast}= p_3^{\ast}\defeq 2$ already in the construction of the instance of the zone tariff design problem with connected zones. 
\end{proof}

We remark that this is also true for the zone-partition subproblems with single counting and with connected or ring zones with the objective of maximizing the revenue in \cite{Otto2017} because there the prices are always set to the same values independent of the instance, as well.

\paragraph{The Price-Setting Subproblem}
We have seen that the zone-partition subproblem remains NP-complete in both cases with and without the requirement of connected zones.
We now study the price-setting subproblem.
Let a zone partition $\cZ$ be given.
We define $\kappa\defeq \max_{d\in D} \sigma(W_d)$ as the maximum number of zones that are traversed along a path.
In general, $\kappa$ is bounded from above by the maximum number $K$ of nodes of a path, which is at most $\vert V \vert$ if the paths are elementary, i.e., $\kappa\leq K \leq \vert V \vert$.
In the single counting case, it holds $\kappa\leq N$.

For $k\in \{1,\ldots,\kappa\}$, we set 
\[D_k \defeq \{d\in D: \sigma(W_d)=k\},\]
which is the set of all OD pairs that traverse $k$ zones.
Only in this step and for the according no-stopover property it is important whether multiple or single counting is considered.
Otherwise, the price-setting subproblem does not differ between the two variants.

The price-setting subproblem without requiring the no-elongation property and the no-stopover property, which we call the \emph{unrestricted price-setting subproblem}, is
\begin{mini}
	{p_k}{\sum_{k=1}^{\kappa} \sum_{d\in D_k} t_d \vert r_d - p_k\vert}{\label{ip:zone-no-properties}}{} 
	\addConstraint{0}{\leq p_k}{\quad\textup{for all } k\in \{1,\ldots,\kappa\}}
	\addConstraint{p_k}{\in \RR}{\quad\textup{for all } k\in \{1,\ldots,\kappa\},}
\end{mini}
which breaks down into a flat tariff design problem (see \Cref{sec:flat-tariff}) for each number of traversed zones~$k$ (see \cite[Thm.~1]{Hamacher2004}) and can hence be solved in linear time by \Cref{prop:fixed-zones-pricing-unconstrained}.

\begin{theorem}
	\label{prop:fixed-zones-pricing-unconstrained}
	A price list $(p_1,\ldots,p_{\kappa})$ is an optimal solution to the price-setting subproblem \eqref{ip:zone-no-properties} if and only if ${p_k \in \weightedMedian_{d\in D_k}(r_d,t_d)}$ if $D_k\neq \emptyset$, and $p_k\in \RR_{\geq 0}$ arbitrary otherwise.
	Thus, the unrestricted price-setting subproblem can be solved in $\cO(D)$.
	
	In particular, there is an optimal solution with $\{p_1,\ldots,p_{\kappa}\} \subseteq \{r_d: d\in D\}$.
\end{theorem}
\begin{proof}
	The unrestricted price-setting subproblem breaks down into a flat tariff design problem for each number of traversed zones~$k\in \{1,\ldots, \kappa\}$.
	As we saw in \Cref{sec:flat-tariff}, for each $k\in \{1,\ldots, \kappa\}$ with $D_k\neq \emptyset$, we have that $p_k$ is an optimal solution to the $k$-th flat tariff design problem if and only if $p_k\in \weightedMedian(D_k)$.
	These values can be computed in $\sum_{k\in \{1,\ldots,\kappa\}} \cO(\vert D_k\vert) = \cO(\vert D\vert)$.
\end{proof}

If the no-elongation property and the no-stopover property are required, which are implemented with their sufficient conditions (see \Cref{prop:zone-properties}, with different constraints for multiple counting [M] and single counting [S]), the price-setting subproblem can be solved in polynomial time with respect to $\vert D \vert$ and $\kappa$ by the following LP formulation:
\begin{mini}
	{p_k,y_d}{\sum_{d\in D} t_d y_d}{\label{lp:price-setting-poly}}{} 
	\addConstraint{r_d-p_k}{\leq y_d} {\quad\textup{for all } d\in D_k, k\in \{1,\ldots,\kappa\}}
	\addConstraint{p_k-r_d}{\leq y_d}{\quad\textup{for all } d\in D_k, k\in \{1,\ldots,\kappa\}}
	\addConstraint{0}{\leq p_1}
	\addConstraint{p_k}{\leq p_{k+1}}{\quad\textup{for all } k\in \{1,\ldots,\kappa-1\}}
	\addConstraint{\textup{[M]}\quad p_k}{\leq p_i + p_{k-i+1}}{\quad\textup{for all } k\in \{1,\ldots, \kappa\},\ i\in \{1,\ldots,k\}}
	\addConstraint{\textup{[S]}\quad p_k}{\leq p_{i_1} + p_{i_2}}{\quad\textup{for all } k\in \{1,\ldots, \kappa\},\ i_1,i_2\in \{1,\ldots,k\}}
	\addConstraint{}{}{\qquad\textup{with } i_1+i_2 \geq k+1}
	\addConstraint{p_k,y_d}{\in \RR}{\quad\textup{for all } k\in \{1,\ldots,\kappa\},\ d\in D,}
\end{mini}
which has $\vert D \vert + \kappa$ variables and for multiple counting $2\cdot \vert D \vert + \kappa + \frac{\kappa \cdot (\kappa+1)}{2}$ constraints and for single counting $2\cdot \vert D \vert + \kappa + \frac{\kappa \cdot (\kappa+1)(\kappa+2)}{6}$.
The number of constraints for the no-stopover property in case of multiple counting can be reduced to at most $\frac{(\kappa-2)(\kappa+1)}{4}$ by only enforcing the constraint for $k\in \{3,\ldots,\kappa\}$, $i\in \{2,\ldots, \lfloor \frac{k+1}{2}\rfloor\}$ as shown in \cite[Thm.~19]{Schoebel2020}. Similarly, the range for the single counting constraints for the no-stopover property can be reduced to $k \in \{3,\ldots, \kappa\}$, $i_1 \in \{\lceil \frac{k+1}{2}\rceil,\ldots, k\}$, $i_2 \in \{k+1-i_1,\ldots,i_1\}$.

Because of the relevance of increasing prices in practice, we have a closer look at the price-setting subproblem enforcing the no-elongation property by its sufficient condition of a monotonically increasing price function (see \Cref{prop:zone-properties}), which we call the \emph{monotone price-setting subproblem}:
\begin{mini}
	{p_k}{\sum_{k=1}^{\kappa} \sum_{d\in D_k} t_d \vert r_d - p_k\vert}{\label{ip:zone-no-elong}}{} 
	\addConstraint{0}{\leq p_1}{}
	\addConstraint{p_k}{\leq p_{k+1}}{\quad\textup{for all } k\in \{1,\ldots,\kappa-1\}}
	\addConstraint{p_k}{\in \RR}{\quad\textup{for all } k\in \{1,\ldots,\kappa\}.}
\end{mini}

This problem still has $\vert D \vert + \kappa$ variables and $2\cdot \vert D \vert + \kappa$ constraints.
Therefore, we aim to find a better solution method, which is motivated by \Cref{prop:zone-no-elong-wmed} and \Cref{remark:price-level}.

For the unrestricted price-setting subproblem, we know by \Cref{prop:fixed-zones-pricing-unconstrained} that in an optimal solution we have $p_k\in \weightedMedian(D_k)$ for all $k\in \{1,\ldots, \kappa\}$ with $D_k\neq\emptyset$.
Note that this does not hold in general for the monotone price-setting subproblem as the following simple example shows:
Let $\kappa\defeq 3$ and set $D_1 \defeq \{d_1\}$, $t_{d_1}\defeq1$, $r_{d_1}\defeq 2$ and $D_2\defeq \{d_2\}$, $t_{d_2}\defeq 2$, $r_{d_2}\defeq 1$ and $D_3 \defeq \{d_3\}$, $t_{d_3}\defeq 3$, $r_{d_3}\defeq 3$.
Then, $(p_1,p_2,p_3)=(1,1,3)$ is an optimal solution to the monotone price-setting subproblem but $1\notin \{2\}=\weightedMedian(D_1)$.

\begin{lemma}
	\label{prop:zone-no-elong-wmed}
	Let $p=(p^{\ast}_1, \ldots, p^{\ast}_{\kappa})$ be an optimal solution to the monotone price-setting subproblem.
	If $p^{\ast}_1 < \ldots < p^{\ast}_{\kappa}$, then $p^{\ast}_k \in \weightedMedian(D_k)$ for all $k\in \{1,\ldots, \kappa\}$.
\end{lemma}
\begin{proof}
	Let $k\in \{1,\ldots,\kappa\}$ with $D_k \neq \emptyset$.
	Assume $p^{\ast}_k \notin \weightedMedian(D_k)$, and let $p' \in \weightedMedian(D_k)$.
	By \Cref{sec:flat-tariff}, it holds that $p'_k$ minimizes $\varphi_k\colon p_k \mapsto \sum_{d\in D_k} t_d\vert r_d - p_k\vert$, and $\varphi_k(p'_k) < \varphi_k(p^{\ast}_k)$.
	Because $p^{\ast}_{k-1}<p^{\ast}_k<p^{\ast}_{k+1}$ (with $p^{\ast}_0\defeq -\infty$ or $p^{\ast}_{\kappa+1}\defeq+\infty$ if necessary), we can increase or decrease $p^{\ast}_k$ towards $p'_k$ so that the order of the prices remains increasing.
	Because $\varphi_k$ is convex, this leads to a reduction of the objective function value, which is a contradiction to $(p^{\ast}_1,\ldots,p^{\ast}_{\kappa})$ being an optimal solution.
\end{proof}

If we consider the example from before, we see that by merging $D_1$ and $D_2$ to a common price level we indeed get $1\in \weightedMedian(D_1\cup D_2)$.
We formalize this idea in \Cref{remark:price-level}.

\begin{remark}
	\label{remark:price-level}
	We can generalize the result of \Cref{prop:zone-no-elong-wmed} by considering distinct price levels as follows:
	Let $0\leq p^{\ast}_1\leq \ldots\leq p^{\ast}_{\kappa}$ be an optimal solution to the monotone price-setting subproblem.
	Let $q_1,\ldots,q_L\in \RR_{\geq 0}$ for some $L\leq \kappa$ be the distinct price levels that satisfy $q_1< \ldots < q_L$ and $\{p_1^{\ast},\ldots,p_{\kappa}^{\ast}\}=\{q_1,\ldots,q_L\}$.
	For all $l\in \{1,\ldots,L\}$, we set $I_l\defeq\left\{k\in\{1,\ldots,\kappa\}: p_k=q_l\right\}$.
	Analogously to \Cref{prop:zone-no-elong-wmed}, it holds that $q_l \in \weightedMedian(\bigcup_{k\in I_l}D_k)$.
\end{remark}

Note however that \Cref{remark:price-level} only yields a necessary but not sufficient condition for optimal solutions of the monotone price-setting subproblem:
In the previous example, we could set $I_1=\{1\}$, $I_2=\{2,3\}$ with $q_1=1$, $q_2=3$.
Then $q_1 < q_2$ and ${q_1\in\weightedMedian(D_1)}$, ${q_2\in\weightedMedian(D_1\cup D_2)}$, however $(p_1,p_2,p_3)=(1,3,3)$ is not an optimal solution.
It is crucial that we only form a common price level for consecutive prices that violate monotonicity if it is not enforced, as we see in the rest of this section.

The relation of optimal solutions to the monotone price-setting subproblem to weighted medians motivates the development of \Cref{alg:zone-no-elong}, which computes an optimal solution in $\cO(\kappa\cdot\vert D \vert)$ as we prove in \Cref{prop:zone-no-elong-alg-correctness}.

\begin{algorithm}[tb]
	\caption{Price-setting subproblem with monotonically increasing prices}
	\label{alg:zone-no-elong}
	\SetKwInOut{Input}{Input}
	\SetKwInOut{Output}{Output}
	
	\Input{Set of OD pairs $D$ with a partition $D_1,\ldots, D_{K}\neq \emptyset$, numbers of passengers $t_d$ and reference prices $r_d$ for all $d\in D$}
	\Output{A monotonically increasing price list $(p_1,\ldots,p_{\kappa})$}
	
	Initialize\\
	$I\gets [\{1\},\ldots,\{\kappa\}]$, \atcp{Index sets} \label{line:alg-zone-no-elong-initialize-I}\\
	$D\gets [D_1,\ldots,D_{\kappa}]$, \atcp{OD pairs} \label{line:alg-zone-no-elong-initialize-D}\\
	$P \gets [p_1,\ldots,p_{\kappa}]$ with $p_k \in \weightedMedian(D_k)$ for all $k\in \{1,\ldots,\kappa\}$, \atcp{Price levels}\label{line:alg-zone-no-elong-initialize-P}\\
	$k \gets 1$. \atcp{Start list indexing at 1}
	
	\While{$k\neq \text{length}(P)$}{
		\tcp{Check whether the prices are increasing}
		\uIf{$P[k] > P[k + 1]$\label{line:alg-zone-no-elong-begin-merge}}{
			\tcp{Merge operation:}
			\tcp{Merge the entries at position $k$ and $k+1$ and store them at position $k$}
			$I[k] \gets I[k] \cup I[k+1]$\\
			$D[k] \gets D[k] \cup D[k+1]$\\
			$P[k] \gets p_k$ with $p_k\in \weightedMedian(D[k])$\\
			\tcp{Delete the entry at position $k+1$ and thus shorten the lists}
			$I$.delete($k+1$), $D$.delete($k+1$), $P$.delete($k+1$)\\
			\If{$k\neq 1$}{
				$k \gets k-1$ \label{line:alg-zone-no-elong-end-merge}
		} }
		\Else(\label{line:alg-zone-no-elong-begin-move}){
			\tcp{Move operation:}
			$k \gets k+1$\label{line:alg-zone-no-elong-end-move}}
	}
	\For{$l= 1,\ldots,\text{length}(P)$\label{line:alg-zone-no-elong-begin-price-list}}{
		\For{$k \in I[l]$}{
			$p_k \gets P[l]$ \label{line:alg-zone-no-elong-end-price-list}
		}
	}
	
	\Return $(p_1,\ldots,p_{\kappa})$
\end{algorithm}

The idea of \Cref{alg:zone-no-elong} is to start with $\kappa$ price levels, one for each number of traversed zones, and to add constraints ensuring monotonicity when it is violated.
The list $P$ of price levels (sorted by number of traversed zones) is checked whether it is increasing.
Every time a price level $p_k$ and its successive price level $p_{k+1}$ in $P$ are not increasing, they are combined to one common price level, meaning that the price must be the same for all OD pairs assigned to this new combined level although they traverse different numbers of zones.
The new price is computed and the list of price levels~$P$ is checked again for monotonicity.
During the algorithm, $D$ stores a list of corresponding OD pairs for every price level, and $I$ stores the corresponding index sets.
The algorithm terminates when the list of price levels is monotonically increasing.

\Cref{alg:zone-no-elong} consists of two main operations, which are performed within the while-loop.
We call \crefrange{line:alg-zone-no-elong-begin-merge}{line:alg-zone-no-elong-end-merge} the \emph{merge} operation and \crefrange{line:alg-zone-no-elong-begin-move}{line:alg-zone-no-elong-end-move} the \emph{move} operation.
For simplicity, we assume that the input sets $D_1,\ldots, D_{\kappa}$ are not empty.
This condition can be achieved by already merging the empty levels to neighboring ones, i.e., with the next lower or higher number of traversed zones.

\begin{example}
	\label{ex:alg}
	Before we prove correctness, we illustrate \Cref{alg:zone-no-elong} with an example.
	The data for the sets of OD pairs traversing a certain number of zones is derived from a PTN with fixed zones.
	To simplify notation, every OD pair has one passenger ($t_d = 1$), and instead of the OD pairs, we here give a list $R$ of reference prices belonging to the OD pairs, and sort all lists.
	The initial state is shown in \Cref{tab:ex-alg-1}.
	The OD pairs traverse between 1 and 6 zones, the reference price within each level is the same and a weighted median is shown in the last row.
	Because $1\leq 3$, a move operation is performed in the first iteration.
	Hence, the lists do not change.
	In the second iteration, however, we have $3>1$, which leads to a merge operation.
	The levels for traversing 2 and 3 zones are combined, resulting in a new common price level of~3, shown in \Cref{tab:ex-alg-2}.
	The next three iterations are again move operations because $1\leq 3$, $3\leq 5$ and $5\leq 6$, so the state does not change.
	In the sixth iteration, a merge operation is necessary because $6>4$, resulting in \Cref{tab:ex-alg-3}.
	Because also $5>4$ in the seventh iteration, another merge operation is performed leading to \Cref{tab:ex-alg-4}.
	Because then $3\leq 4$, the while-loop terminates and $(1,3,3,4,4,4)$ is returned as the final price list.
	\begin{table}
		\centering
		\begin{subtable}{0.472\textwidth}
			\centering
			\begin{tabular}[t]{lcccccc}
				\toprule
				$I$ & $[1]$ & $[2]$ & $[3]$ & $[4]$ & $[5]$ & $[6]$\\
				\addlinespace
				$R$ &$[1]$ & $[3,3]$ & $[1]$ & $[5]$ & $[6,6]$ & $[4,4,4,4]$\\
				\addlinespace
				$P$ & 1 & 3 & 1 & 5 & 6 & 4\\
				\bottomrule
			\end{tabular}
			\caption{State after the initialization}
			\label{tab:ex-alg-1}
		\end{subtable}%
		\qquad
		\begin{subtable}{0.472\textwidth}
			\centering
			\begin{tabular}[t]{lccccc}
				\toprule
				$I$ & $[1]$ & $[2,3]$ & $[4]$ & $[5]$ & $[6]$\\
				\addlinespace
				$R$ &$[1]$ & $[1,3,3]$ & $[5]$ & $[6,6]$ & $[4,4,4,4]$\\
				\addlinespace
				$P$ & 1 & 3 & 5 & 6 & 4\\
				\bottomrule
			\end{tabular}
			\caption{State after the second iteration}
			\label{tab:ex-alg-2}
		\end{subtable}\\[1em]
		\begin{subtable}[t]{0.45\textwidth}
			\centering
			\begin{tabular}{lcccc}
				\toprule
				$I$ & $[1]$ & $[2,3]$ & $[4]$ & $[5,6]$ \\
				\addlinespace
				$R$ &$[1]$ & $[1,3,3]$ & $[5]$ & $[4,4,4,4,6,6]$\\
				\addlinespace
				$P$ & 1 & 3 & 5 & 4\\
				\bottomrule
			\end{tabular}
			\caption{State after the sixth iteration}
			\label{tab:ex-alg-3}
		\end{subtable}%
		\quad
		\begin{subtable}[t]{0.45\textwidth}
			\centering
			\begin{tabular}{lccc}
				\toprule
				$I$ & $[1]$ & $[2,3]$ & $[4,5,6]$ \\
				\addlinespace
				$R$ &$[1]$ & $[1,3,3]$ & $[4,4,4,4,5,6,6]$\\
				\addlinespace
				$P$ & 1 & 3 & 4\\
				\bottomrule
			\end{tabular}
			\caption{State after the seventh iteration and final state.}
			\label{tab:ex-alg-4}
		\end{subtable}%
		\caption{States during performing \Cref{alg:zone-no-elong} in \Cref{ex:alg}.}
		\label{tab:ex-alg}
	\end{table}	
\end{example}

We prepare the correctness proof of \Cref{alg:zone-no-elong} with the auxiliary \Cref{prop:zone-optimal-solution-with-equality}.
\begin{lemma}
	\label{prop:zone-optimal-solution-with-equality}
	Let $(p'_1,\ldots,p'_{\kappa})$ be an optimal solution to the unrestricted price-setting subproblem.
	If $p'_k > p'_{k+1}$ for some $k\in \{1,\ldots,\kappa-1\}$, then there is an optimal solution $(p_1^{\ast},\ldots,p_{\kappa}^{\ast})$ to the monotone price-setting subproblem that has $p_k^{\ast}=p_{k+1}^{\ast}$.
\end{lemma}
\begin{proof}
	Let $k\in \{1,\ldots,\kappa-1\}$ with $p'_k > p'_{k+1}$ for an optimal solution $(p'_1,\ldots,p'_{\kappa})$ to the unrestricted price-setting subproblem be arbitrary but fixed.
	Let $(p_1^{\ast},\ldots,p_{\kappa}^{\ast})$ be an optimal solution to the monotone price-setting subproblem.
	If $p_k^{\ast}=p_{k+1}^{\ast}$, we are done.
	So now assume that $p^{\ast}_k<p^{\ast}_{k+1}$.
	We show that we can modify the solution until $p_k^{\ast}=p_{k+1}^{\ast}$.
	For $i\in \{k,k+1\}$, we have that $p'_i$ is a minimum of $\varphi_i\colon p_i \mapsto \sum_{d\in D_i} t_d\vert r_d - p_i\vert$ and we can increase/decrease $p_i^{\ast}$ towards~$p'_i$ without worsening the objective function value by \Cref{sec:flat-tariff}.	
	\begin{figure}
		\centering
		\begin{subfigure}[t]{0.34\textwidth}
			\centering
			\begin{tikzpicture}[decoration=brace]
				\def\eps{0.1}
				\draw[->, thick] (0,0)--(5,0);
				\foreach \x/\xtext in {0/0, 1/$p'_{k+1}$,2/$p'_k$,3/$p_k^{\ast}$,4/$p_{k+1}^{\ast}$} \draw(\x,5pt)--(\x,-5pt) node[below] {\xtext};
				\draw[->,blue, thick] (4-\eps,0.2)--(3+\eps,0.2);
			\end{tikzpicture}
			\caption{}
			\label{fig:case-distinction-order-1}
		\end{subfigure}
		\begin{subfigure}[t]{0.3\textwidth}
			\centering
			\begin{tikzpicture}[decoration=brace]
				\def\eps{0.1}
				\draw[->, thick] (0,0)--(4,0);
				\foreach \x/\xtext in {0/0, 1/$p_k^{\ast}$,2/$p'_k$,3/$p_{k+1}^{\ast}$} \draw (\x,5pt)--(\x,-5pt) node[below] {\xtext};
				\draw[decorate, yshift=2ex]  (0,0) -- node[above=0.4ex] {$p'_{k+1}$}  (2,0);
				\draw[->,blue, thick] (3-\eps,0.2)--(2+\eps,0.2);
				\draw[->,blue, thick] (1+\eps,0.2)--(2-\eps,0.2);
			\end{tikzpicture}
			\caption{}
			\label{fig:case-distinction-order-2}
		\end{subfigure}
		\begin{subfigure}[t]{0.3\textwidth}
			\centering
			\begin{tikzpicture}[decoration=brace]
				\def\eps{0.1}
				\draw[->, thick] (0,0)--(4,0);
				\foreach \x/\xtext in {0/0, 1/$p_k^{\ast}$,2/$p_{k+1}^{\ast}$,3/$p'_k$} \draw(\x,5pt)--(\x,-5pt) node[below] {\xtext};
				\draw[decorate, yshift=2ex]  (0,0) -- node[above=0.4ex] {$p'_{k+1}$}  (3,0);
				\draw[->,blue, thick] (1+\eps,0.2)--(2-\eps,0.2);
			\end{tikzpicture}
			\caption{}
			\label{fig:case-distinction-order-3}
		\end{subfigure}
		\caption{Case distinction of the orders of the values $p'_k, p'_{k+1},p_k^{\ast},p_{k+1}^{\ast}$ in \Cref{prop:zone-optimal-solution-with-equality}.}
		\label{fig:case-distinction-order}
	\end{figure}
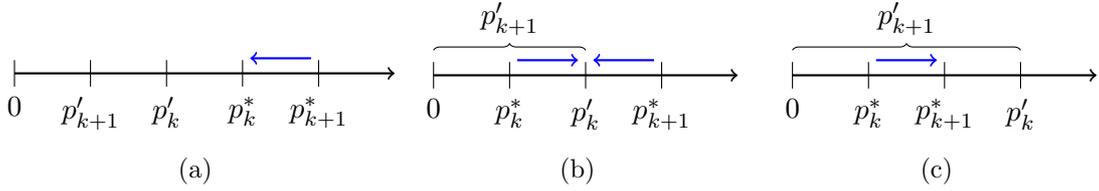
	The following cases can occur:
	\begin{itemize}
		\item If $p'_k \leq p_k^{\ast}$, then we obtain the order depicted in \Cref{fig:case-distinction-order-1}. 
		We decrease $p_{k+1}^{\ast}$ to $p_k^{\ast}$.
		\item If $p_k^{\ast} < p'_k < p_{k+1}^{\ast}$, we obtain the order depicted in \Cref{fig:case-distinction-order-2}.
		We increase $p_k^{\ast}$ to $p'_k$ and decrease $p_{k+1}^{\ast}$ to $p'_k$.
		\item If $p_k^{\ast} < p_{k+1}^{\ast} \leq p'_k$, we obtain the order depicted in \Cref{fig:case-distinction-order-3}.
		We increase $p_k^{\ast}$ to $p_{k+1}^{\ast}$.
	\end{itemize}
	Because we only move $p^{\ast}_k$ and $p^{\ast}_{k+1}$ towards each other, the price list remains increasing.
	Hence, there is an optimal solution to the monotone price-setting subproblem with $p^{\ast}_k=p^{\ast}_{k+1}$.
\end{proof}

\begin{theorem}
	\label{prop:zone-no-elong-alg-correctness}
	\Cref{alg:zone-no-elong} solves the price-setting subproblem~\eqref{ip:zone-no-elong} (requiring monotonically increasing prices, thus satisfying the no-elongation property) in $\cO(\kappa\cdot \vert D\vert)$.
\end{theorem}
\begin{proof}	
	Running time: In every iteration of the while-loop, either a merge or a move operation is performed.
	Merge can be performed at most $\kappa-1$ times because the length of the lists $P, D$ and $I$ is reduced by 1 by every merge operation.
	Move also is performed at most $\kappa-1$ times in total because the number of levels to still look at, which is $\text{length}(P)-k$, is reduced by 1 by every move operation and is also not increased by merge operations.
	Indeed, in every merge operation, the difference is decreased by 1 if $k=1$ and remains the same if $k>1$.
	Hence, \Cref{alg:zone-no-elong} terminates after at most $2\kappa-2$ iterations.
	The theoretical running time is composed of $\cO(\vert D \vert)$ for the initialization in \cref{line:alg-zone-no-elong-initialize-D,line:alg-zone-no-elong-initialize-I,line:alg-zone-no-elong-initialize-P}, $(\kappa-1)\cdot \cO(\vert D\vert)$ for at most $\kappa-1$ merge operations, $(\kappa-1)\cdot \cO(1)$ for at most $\kappa-1$ move operations and $\cO(\kappa)$ for setting the final price list in \crefrange{line:alg-zone-no-elong-begin-price-list}{line:alg-zone-no-elong-end-price-list}.
	This yields a total running time of $\cO(\kappa\cdot\vert D \vert)$.
	
	Correctness: The aim is to solve problem~\eqref{ip:zone-no-elong}.
	In \cref{line:alg-zone-no-elong-initialize-P} of \Cref{alg:zone-no-elong}, the prices are set to weighted medians, which is an optimal solution to the relaxed problem~\eqref{ip:zone-no-properties} by \Cref{prop:fixed-zones-pricing-unconstrained}.
	If the initialized price list $P=(p'_1,\ldots,p'_{\kappa})$ is increasing, this yields an optimal solution also to problem~\eqref{ip:zone-no-elong}, and the algorithm terminates after several move operations without changing the price list.
	Therefore, we consider now the case that it is not increasing.
	Let $k$ be minimal such that $p'_k > p'_{k+1}$.
	By \Cref{prop:zone-optimal-solution-with-equality}, there is an optimal solution $(p_1^{\ast},\ldots,p_{\kappa}^{\ast})$ to problem~\eqref{ip:zone-no-elong} with $p^{\ast}_k=p^{\ast}_{k+1}$.
	Therefore, we can ensure monotonicity by adding the constraint $p_k = p_{k+1}$ to the problem formulation.
	Equivalently, this means we condense the variables $p_k$ and $p_{k+1}$ to a common variable $p_{k,k+1}$ with $D_{k,k+1}=D_k \cup D_{k+1}$, hence reducing the number of variables (price levels) by one.
	In \Cref{alg:zone-no-elong}, this is implemented in form of the merge operation.
	Computing the new median, the new list of price levels $P$ is an optimal solution to problem~\eqref{ip:zone-no-properties} with the condensed input.
	This process is repeated.
	If we always started checking monotonicity from the beginning of the list $P$, we would be done.
	However, we can do a bit better. 
	Because we always search for the smallest level for which monotonicity is violated, it suffices to decrease the index that is currently looked at by 1 in \cref{line:alg-zone-no-elong-end-move}.
	This is because, when we perform a merge operation for index $k$, nothing is changed for smaller indices $k'$ with $k'\leq k-2$, and those prices are still increasing after the iteration.
	Hence, upon termination, \Cref{alg:zone-no-elong} returns an optimal solution to the price-setting subproblem with increasing prices.	
\end{proof}

\begin{remark}
	For fixed $\kappa$, \Cref{alg:zone-no-elong} is a linear-time algorithm in $\cO(\vert D\vert)$.
	Note that the solution method of \cite{Zemel1984} also can be applied to the LP formulation of \eqref{ip:zone-no-elong} if $\kappa$ is fixed analogously to the proof \Cref{prop:distance-linear} but with $s=\kappa$, solving the problem in $\cO(\vert D\vert)$ as well.
\end{remark}

\subsection{Mixed-Integer Linear Programming Formulation}
\label{sec:zd-mip}
We now provide mixed-integer linear programming (MILP) formulations for \ZDXY{} with $\textup{X} \in \{\textup{M, S}\}$ and  $\textup{Y} \in \{\textup{A,C}\}$ with constraints ensuring the no-elongation property and the no-stopover property.
It extends the formulation of \cite{Otto2017} for connected zones and single counting to all four cases of \Cref{tab:four-variants} and can also ensure the no-elongation property and the no-stopover property by implementing their sufficient conditions (see \Cref{prop:zone-properties}). 

According to \Cref{prop:bounded-prices}, we define $\bar{r}\defeq \max\{r_d: d\in D\}$, and the maximum number of zones traversed on a path is at most 
$K \defeq \begin{cases}
	\max_{d\in D}\vert V(W_d)\vert & \text{if multiple counting,}\\
	\min\{N,\max_{d\in D}\vert V(W_d)\vert\} &\text{if single counting.}
\end{cases}$

As in \cite{Otto2017} (except for renaming), the following variables are used, where only the variables for multiple counting are new:
\begin{itemize}
	\item a binary variable $x_{vz}\in \{0,1\}$ for all stations $v\in V$ and zones $z \in \{1,\ldots,N\}$ that is 1 if and only if station $v$ is assigned to zone $Z_z$,
	\item {[only for connected zones:]}
	\begin{itemize}
		\item a continuous variable $f_{v_1 v_2}\in \RR_{\geq 0}$ for all stations $v_1,v_2\in V\!$,
		\item a binary variable $s_v\in \{0,1\}$ for all stations $v\in V\!$,
		\item a continuous variable $f_{0v}\in \RR_{\geq 0}$ for all $v\in V\!$,
	\end{itemize}
	\item {[only if multiple counting:]} a binary variable $b_e\in \{0,1\}$ for all edges $e=\{v_1,v_2\}\in E$ that is 1 if and only if the stations $v_1$ and $v_2$ are in different zones,
	\item {[only if single counting:]} a binary variable $b^z_d\in \{0,1\}$ for all OD pairs $d\in D$ and zones $z\in \{1,\ldots,N\}$ that is 1 if and only if the path of OD pair $d$ traverses zone $Z_z$,
	\item a binary variable $c^k_d\in \{0,1\}$ for all OD pairs $d\in D$ and numbers of traversed zones $k\in \{1,\ldots, K\}$ that is 1 if and only if the path of OD pair $d$ traverses exactly $k$ zones,
	\item a continuous variable $p_k$ for all numbers of traversed zones $k\in \{1,\ldots,K\}$ that denotes the price for traversing $k$ zones,
	\item a continuous variable $\pi_d$ for all OD pairs $d\in D$ that denotes the price for traveling for OD pair $d$,
	\item a continuous variable $y_d\in \RR$ for all OD pairs $d\in D$ for linearizing the objective function.
\end{itemize} 

We present the MILP formulation in a modular way, where constraints \labelcref{ip:station-assignment,ip:connected-zones,ip:single-counting} have previously been used in \cite{Otto2017}.
An explanation of the constraints is given below.

\noindent
\begin{tabular}{C{\textwidth-2\tabcolsep}}
	\midrule
	objective function\\
	\midrule
\end{tabular}
\begin{equation} \label{ip:objective} 
	\begin{alignedat}{3} 
		&\min 	& \sum_{d\in D} t_d \cdot y_d& &\quad&\\
		&\text{ s.t. } 	& \pi_d - r_d &\leq y_d &&\textup{for all } d\in D\\
		&& r_d - \pi_d &\leq y_d &&\textup{for all } d\in D
	\end{alignedat}
\end{equation}
\begin{tabular}{C{\textwidth-2\tabcolsep}}
	\midrule
	station assignment\\
	\midrule
\end{tabular}
\begin{alignat}{3}
	&&\sum_{z=1}^N x_{vz}& =1 &\quad& \textup{for all } v \in V \label{ip:station-assignment} 
\end{alignat}
\begin{tabular}{C{\textwidth-2\tabcolsep}}
	\midrule
	connected zones (optional)\\
	\midrule
\end{tabular}
\begin{subequations} \label{ip:connected-zones} 
	\begin{alignat}{3}
		&&f_{0v} &\leq s_v \cdot \vert V \vert &\quad&\textup{for all } v \in V \label{ip:connected-zones-1}\\
		&&s_{v_1}+s_{v_2} + x_{v_1 z} + x_{v_2 z}&\leq 3 && \textup{for all } v_1,v_2 \in V, v_1\neq v_2, z \in \{1,\ldots,N\}\label{ip:connected-zones-2}\\
		&&f_{v_1v_2}&=0 && \textup{for all } \{v_1,v_2\} \not \in E\label{ip:connected-zones-3}\\
		&&f_{v_1 v_2}&\leq (1+x_{v_1 z} - x_{v_2 z})\cdot \vert V \vert && \textup{for all } \{v_1,v_2\} \in E, z\in \{1,\ldots,N\}\label{ip:connected-zones-4}\\
		&&\sum_{v_2\in V \cup \{0\}} f_{v_2 v_1}&= 1+ \sum_{v_2\in V} f_{v_1 v_2} && \textup{for all } v_1 \in V \label{ip:connected-zones-5}
	\end{alignat}
\end{subequations}
\begin{tabular}{C{\textwidth-2\tabcolsep}}
	\midrule
	multiple counting (Alternative 1)\\
	\midrule
\end{tabular}
\begin{subequations} \label{ip:multiple-counting} 
	\begin{alignat}{3}
		&&x_{v_1 z}-x_{v_2 z} &\leq b_e&\quad& \textup{for all } e=\{v_1,v_2\}\in E, z\in \{1,\ldots,N\} \label{ip:multiple-counting-1}\\
		&&b_e &\leq 2 - x_{v_1 z} - x_{v_2 z}&& \textup{for all } e=\{v_1,v_2\}\in E, z\in \{1,\ldots,N\} \label{ip:multiple-counting-2}\\
		&&\sum_{k=1}^{K} c^k_d &= 1&& \textup{for all } d \in D \label{ip:multiple-counting-3}\\
		&&\sum_{e\in E(W_d)} b_e &= \sum_{k=1}^{K} (k-1) \cdot c^k_d && \textup{for all } d \in D \label{ip:multiple-counting-4}
	\end{alignat}
\end{subequations}
\begin{tabular}{C{\textwidth-2\tabcolsep}}
	\midrule
	single counting (Alternative 2)\\
	\midrule
\end{tabular}
\begin{subequations} \label{ip:single-counting} 
	\begin{alignat}{3}
		&&\sum_{v \in V(W_d)} x_{vz}&\leq b^z_d \cdot \vert V \vert&\quad& \textup{for all } d \in D, z\in\{1,\ldots,N\} \label{ip:single-counting-1}\\
		&&b^z_d&\leq \sum_{v \in V(W_d)} x_{vz}&& \textup{for all } d \in D, z\in\{1,\ldots,N\} \label{ip:single-counting-2}\\
		&&\sum_{k=1}^{K} c_d^k&= 1&& \textup{for all } d \in D \label{ip:single-counting-3}\\
		&&\sum_{z=1}^N b_d^z&= \sum_{k=1}^{K} k \cdot c^k_d&& \textup{for all } d\in D \label{ip:single-counting-4}
	\end{alignat}
\end{subequations}
\begin{tabular}{C{\textwidth-2\tabcolsep}}
	\midrule
	price assignment\\
	\midrule
\end{tabular}
\begin{subequations} \label{ip:price-assignment} 
	\begin{alignat}{3}
		&&\pi_d&\leq p_k + (1-c^k_d) \cdot \bar{r}&\quad& \textup{for all } d\in D, k \in \{1,\ldots, K\} \label{ip:price-assignment-1}\\
		&&p_k&\leq \pi_d + (1-c^k_d) \cdot \bar{r}&& \textup{for all } d\in D, k \in \{1,\ldots, K\} \label{ip:price-assignment-2}\\
		&& 0&\leq p_k&&\textup{for all } k\in \{1,\ldots,K\} \label{ip:price-assignment-3}
	\end{alignat}
\end{subequations}
\begin{tabular}{C{\textwidth-2\tabcolsep}}
	\midrule
	no-elongation property (optional)\\
	\midrule
\end{tabular}
\begin{alignat}{3}
	&&p_k&\leq p_{k+1} &\quad& \textup{for all } k \in \{1,\ldots,K-1\} \label{ip:price-assignment-4}
\end{alignat}
\begin{tabular}{C{\textwidth-2\tabcolsep}}
	\midrule
	no-stopover property (optional)\\
	\midrule
\end{tabular}
\begin{subequations} \label{ip:no-stopover} 
	\begin{alignat}{3}
		&\text{\small[M]\quad}&p_k&\leq p_i + p_{k-i+1}&\quad& \textup{for all } k\in \{3,\ldots,K\}, i\in \left\{2,\ldots,\left\lfloor \nicefrac{k+1}{2} \right\rfloor \right\} \label{ip:price-assignment-6}\\
		&\text{\small[S]}&p_k&\leq p_{i_1} + p_{i_2} &&\textup{for all } k,i_1, i_2 \in \{1,\ldots, K\} \text{ with } i_1,i_2\leq k, i_1+i_2 \geq k+1  \label{ip:price-assignment-5}
	\end{alignat}
\end{subequations}

\begin{tabular}{C{\textwidth-2\tabcolsep}}
	\midrule
	variable domains\\
	\midrule
\end{tabular}
\begin{equation} \label{ip:variable-domains} 
	\begin{alignedat}{3}
		&&x_{vz}, s_v, b^k_d, b_e, c^k_d&\in \{0,1\}&\quad& \textup{for all } v\in V, e\in E, z\in \{1,\ldots,N\}, d\in D, k\in \{1,\ldots,K\}\\
		&&f_{0v_1}, f_{v_1 v_2}&\in \RR_{\geq 0}&& \textup{for all } v,w \in V\\
		&&p_k,\pi_d, y_d&\in \RR&& \textup{for all } k\in \{1,\ldots, K\}, d\in D
	\end{alignedat}
\end{equation}

The objective function to minimize $\sum_{d\in D} t_d \vert \pi_d - r_d\vert$ is linearized in the constraints~\eqref{ip:objective}.
The constrains~\eqref{ip:station-assignment} specify that each station is assigned to exactly one zone.

Constraints for connected subgraphs based on single- or multi-commodity flows or cuts are applied in many areas besides fare planning, for example, the connected $k$-cut problem \cite{Hojny2021}, forest planning and wildlife conservation \cite{Conrad2007,Dilkina2010, Carvajal2013} and price zones of electricity markets \cite{Grimm2017,Kleinert2019}.
Recently, \cite{Borndoerfer2023} (for vertex covering with capacitated trees) and \cite{Validi2022} (for political districting) gave an overview on different formulations for connected subgraphs and their performances, where single-commodity flows performed well.
The single-commodity flow constraints~\eqref{ip:connected-zones} to ensure connected zones in this paper have been adopted from \cite{Otto2017}.
The idea is to model a flow from an additional source 0 to each station.
It is not allowed to cross zone borders.
Flow starting from the source 0 can only be sent to stations that are assigned to the source \eqref{ip:connected-zones-1}, and at most one station per zone is assigned to the source \eqref{ip:connected-zones-2}.
Flow can only be sent along edges of the PTN \eqref{ip:connected-zones-3} and only if the stations of an edge belong to the same zone \eqref{ip:connected-zones-4}.
To see this, let an edge $\{v_1,v_2\}\in E$ be given.
If there is some $z\in \{1,\ldots,N\}$ such that $v_1,v_2\in Z_z$, i.e., $x_{v_1 z}=x_{v_2 z}=1$, then $1+x_{v_1 z'}-x_{v_2 z'} = 1$ for all $z'\in \{1,\ldots,N\}$ and $f_{v_1 v_2}$ is only restricted by the number of stations $\vert V \vert$.
On the other hand, if $v_1$ and $v_2$ are not in the same zone, then there is some $z\in \{1,\ldots,N\}$ such that $1+x_{v_1 z}-x_{v_2 z} = 0$, and the flow $f_{v_1 v_2}$ is set to 0.
Flow conservation with a demand of one flow unit is modeled in the constraints~\eqref{ip:connected-zones-5}.
Hence, because each station needs to receive one unit of flow and at most one station per zone gets flow from the source and the flow cannot be sent across zone borders, it is enforced that zones are connected.
The constraints~\eqref{ip:connected-zones} can be omitted if connected zones are not required.

The constraints~\eqref{ip:multiple-counting} and~\eqref{ip:single-counting} determine the counting variables for the multiple counting and single counting case, respectively.
In case of multiple counting, a variable $b_{v_1 v_2}$ corresponding to an edge $\{v_1,v_2\}\in E$ is set to 1 if the stations $v_1$ and $v_2$ belong to different zones \eqref{ip:multiple-counting-1}, and to 0 if they are in the same zone \eqref{ip:multiple-counting-2}.
The constraints~\eqref{ip:multiple-counting-3} and~\eqref{ip:multiple-counting-4} count the number of zone border crossings and set the variable $c_d^k$ for each OD pair $d\in D$ to 1 if the path $W_d$ crosses $k-1$ zone borders, which means that it traverses $k$ zones.
In case of single counting, for all OD pairs $d\in D$ and zones $z\in\{1,\ldots,N\}$, the variable $b_d^z$ is set to 1 if there is a station along the path $W_d$ that is assigned to zone $Z_z$ \eqref{ip:single-counting-1}.
It is set to 0 otherwise \eqref{ip:single-counting-2}.
The constraints~\eqref{ip:single-counting-3} and~\eqref{ip:single-counting-4} determine the total number of different zones that are traversed by each OD pair $d\in D$ by setting the corresponding variable $c_d^k$ to 1 if $k$ different zones are traversed and 0 otherwise.

Based on the number of traversed zones, the price of an OD pair $d\in D$ is assigned to the according price level by the constraints~\eqref{ip:price-assignment-1} and~\eqref{ip:price-assignment-2}.
If $c_d^k=1$ for some $d\in D$ and ${k\in \{1,\ldots,K}\}$, then the constraints resolve to $\pi_d = p_k$.
For $c_d^k = 0$, the constraints are $\pi_d \leq p_k + \bar{r}$ and $p_k \leq \pi_d + \bar{r}$, which poses no restriction due to \Cref{prop:bounded-prices}.
While the prices should always be non-negative \eqref{ip:price-assignment-3}, the constraints ensuring the no-elongation property \eqref{ip:price-assignment-4} and the no-stopover property in case of multiple counting \eqref{ip:price-assignment-6} or single counting \eqref{ip:price-assignment-5} are optional.
Finally, the constraints~\eqref{ip:variable-domains} state the variable domains.

\section{Conclusion and Further Research}
\label{sec:conclusion}

In this paper, we give a detailed description and analysis of fare structure design problems concerning flat, affine distance and zone tariffs with the objective function to minimize the sum of absolute deviations between the newly set prices and given reference prices.
For zone tariffs, we distinguish between two pricing options specifying the counting of zones, namely multiple counting (e.g., Boston) and single counting (e.g., Greater Copenhagen).
Further, we consider connected zones as well as the no-elongation property and the no-stopover property.
While the flat and affine distance tariff design problems can be solved in polynomial, even linear time, the zone tariff design problems are NP-hard, even if the price function is already fixed.
On the other hand, the price-setting subproblem, which aims to determine an optimal price function given a fixed zone partition, can be solved in polynomial time.
Moreover, we analyze properties of the different zone tariff design problem variants and present a (mixed-integer) linear programming formulation for each of the fare structure design problems.

The fare structure design problems have been added to the open-source software library LinTim \cite{Schiewe2023, Schiewe} and can be solved with the implemented MILP formulations presented here. 
First tests show that the flat and affine distance tariff design problems can be solved quickly even for large instances.
For the zone tariff design problem, the running time using the solver Gurobi (\cite{GurobiOptimization2024}) grows quickly from few minutes for an instance with 15~stations to instances with 20 to 30 stations that cannot be solved in reasonable time.
Therefore, in addition to solving the MILP with a solver like Gurobi, an interesting working direction is the design of a solution method for the zone tariff design problem based on Benders decomposition or a branch and bound algorithm motivated by \cite{Azadian2018}.
Also symmetry breaking might help \cite{Sherali2001,Margot2010, Pfetsch2019}.

In this paper, we assumed that the route choice of passengers is fixed reflecting that the main decision criterion is the travel time and not the price of a journey.
However, changes in the fare structure and the price level could lead to changes in the preferred paths.
Integrating route choice would be interesting for future research to further concern the passenger perspective during fare planning.
Also the demand could vary based on the price that has to be paid per OD pair as considered in \cite{Schiewe2024a}.
Lowering prices could increase the demand and be an incentive to use sustainable transport modes.

\bibliographystyle{alphaurl}
\bibliography{tariff_design}
	
\end{document}